\documentclass[12pt,a4paper,reqno]{amsart}
\usepackage{relsize, enumerate}
\usepackage{hyperref}\hypersetup{linktocpage}

\usepackage{amsmath,amsfonts,amssymb,amsthm,amscd}
\usepackage{mathtools}
\usepackage{MnSymbol}
\usepackage{pst-node,pstricks,pst-tree}
\usepackage[subrefformat=parens]{subcaption}
\usepackage{xcolor}
\usepackage{comment}
\usepackage{graphicx}
\usepackage[all]{xy}
\usepackage{tikz}

\usepackage{geometry}
\makeatletter
\newcommand*{\rom}[1]{\expandafter\@slowromancap\romannumeral #1@}
\makeatother

    \newcommand{\BA}{{\mathbb {A}}} 
    \newcommand{\BC}{{\mathbb {C}}} 
     
    \newcommand{\BG}{{\mathbb {G}}}

     \newcommand{\BN}{{\mathbb {N}}}
     \newcommand{\BP}{{\mathbb {P}}}
    \newcommand{\BQ}{{\mathbb {Q}}} \newcommand{\BR}{{\mathbb {R}}}

     \newcommand{\BZ}{{\mathbb {Z}}}

     \renewcommand{\CD}{{\mathcal {D}}}

     \newcommand{\CL}{{\mathcal {L}}}
    \newcommand{\CM}{{\mathcal {M}}} 
    \newcommand{\CO}{{\mathcal {O}}}

    \newcommand{\Aut}{{\mathrm{Aut}}}

    \newcommand{\cond}{\mathrm{cond}}

    \newcommand{\dR}{{\mathrm{dR}}}

    \newcommand{\Gal}{{\mathrm{Gal}}} \newcommand{\GL}{{\mathrm{GL}}}

    \renewcommand{\Im}{{\mathrm{Im}}}

    \newcommand{\ord}{{\mathrm{ord}}}

    \newcommand{\Spec}{{\mathrm{Spec}}}

    \newcommand{\wt}{\widetilde}
    \newcommand{\wh}{\widehat}
    
    \newcommand{\pair}[1]{\langle {#1} \rangle}

    \newcommand{\ov}{\overline}

    \newcommand{\ra}{\rightarrow} 
    \newcommand{\bs}{\backslash}

\theoremstyle{plain}
\newtheorem{thm}{Theorem}[section] \newtheorem{cor}[thm]{Corollary}
\newtheorem{lem}[thm]{Lemma}  \newtheorem{prop}[thm]{Proposition}
\newtheorem {conj}[thm]{Conjecture}

\theoremstyle{definition}

\theoremstyle{plain}
\newtheorem{thmL}{Theorem}

\theoremstyle{remark} 

\theoremstyle{remark} \newtheorem{remark}[thm]{Remark}
\theoremstyle{remark} \newtheorem{example}[thm]{Example}

    \numberwithin{equation}{section}

\begin{document}

\title{Mahler measure, motivic regulators and Dirichlet $L$-values}

\author{Wei He}
\address{He: School of Mathematics and Statistics, Xi'an Jiaotong University, Xi'an 710049, P.R. China.} 
\email{hewei0714@xjtu.edu.cn}

\author{Jungwon Lee}
\address{Lee: Max Planck Institute for Mathematics, Vivatsgasse 7, 53111 Bonn, Germany}
\email{jungwon@mpim-bonn.mpg.de}

\date{\today}

\begin{abstract}

Inspired by the work of Deninger, we present a formula that relates the Mahler measure of a two-variable variant of cyclotomic polynomial to regulator of class in motivic cohomology associated to cyclotomic fields and linear combination of special values of the derivative of Dirichlet $L$-functions. The formula is derived by studying the Beilinson regulator map applied to systematically constructed elements in the motivic cohomology group. Under linear independence hypothesis on the derivative of partial Dirichlet $L$-values at $s=0$ and $-1$, we study a Galois module structure of the relevant motivic cohomology and obtain the refined identity for a single $L$-value.
\end{abstract}

\maketitle
\setcounter{tocdepth}{1}
\tableofcontents

\section{Introduction}

The Mahler measure is defined, for a Laurent polynomial $f \in \BC[x_1^\pm, \cdots, x_n^\pm]$ with complex coefficients, by
\begin{equation} \label{def:mahler:meas}
m(f)= \frac{1}{(2 \pi i)^n} \int_{T^n} \log | f(x_1, \cdots, x_n)| \frac{dx_1}{x_1} \cdots \frac{dx_n}{x_n}
\end{equation}
where $T^n$ is the real $n$-torus. It has been regarded as a fundamental quantity, which is a version of the height of its zero locus. The Mahler measure naturally arises from the diverse contexts in arithmetic geometry
or dynamical system, for instance as a regulator or a topological entropy. 

Historically, there have been conjectural links between the Mahler measure and special values of $L$-functions. One of the first examples goes back to Smyth \cite{Smyth:81}. Applying Jensen's formula, he found a connection between the Mahler measure of Chebyshev polynomials and linear sum of Dirichlet $L$-values. For example, 
\begin{align*} \label{ex:mah:chi3}
m((x_1+x_2)^2\pm 3)&=\frac{2}{3}\log 3+\frac{\sqrt{3}}{\pi}L(2,\chi_3),
\end{align*}
where $\chi_3$ is a unique non-trivial Dirichlet character modulo $3$. There are also some numerical experiments by Boyd \cite{Boyd} for polynomials attached to elliptic curves. For example, a computation suggests
\begin{equation}\label{ec}m(x_1^{-1}+x_2^{-1}+1+x_1+x_2)=\frac{15}{(2\pi)^2}L(2,E)
\end{equation}
where $E$ is the elliptic curve defined by taking the projective closure of a zero locus of given polynomial.

The theoretical evidence underlying such phenomena was extensively settled by Deninger \cite{Den:97}. In view of Beilinson's conjecture on the relation between the regulator map and special $L$-values, Deninger gave a general geometric interpretation of the Mahler  measure $m(f)$ as a regulator of relevant motivic cohomology class {with the main boundary condition that $f(x_1,\cdots,x_{n-1},0)$ does not vanish on $T^{n-1}$. Accordingly, he showed that \eqref{ec} holds up to a non-zero rational factor predicted by the Beilinson conjecture. Later on, Besser--Deninger \cite{Bes:Den} further expressed $m(f)$ for a class of polynomials $f \in \mathbb{Q}[x_1,x_2]$ defining a CM elliptic curve $E$ over $\BQ$, in terms of regulator and $L'(0,E)$ by adopting the work of Coleman--de Shalit \cite{coleman1988p}.

In this article, we present a $\GL_1$-analogue of the result of this type, that is, we adopt the method and obtain a formula linking the Mahler measure of a two-variable variant of cyclotomic polynomial whose zero locus is $\BA_{\BZ[\mu_N]^+}^1$  to the regulator of systematically constructed relevant cohomology class and the derivative of modulo $N$ Dirichlet $L$-values at $s=0,-1$. However, we emphasise that we did not consider here the $p$-adic analogue, which is also included in  the suggestion \cite[p.20]{Bes:Den}. 

Under the linear independence hypothesis among the derivatives of partial Dirichlet $L$-function at $s=0,-1$, we study a Galois module structure of the relative motivic cohomology and obtain a special value formula relating \[L'\left(\frac{\chi(-1)-1}{2},\chi\right)\] 
to $\chi$-part of the constructed cohomology class that matches with the known cases of Beilinson's conjecture. 


\subsection{Main results}
\label{ms}

Here and throughout, we denote by $H_{\text{sing}}^{\tiny{\bullet}}$, $H_\CM^{\tiny{\bullet}}$ and $H_\CD^{\Tiny{\bullet}}$ the singular, motivic and Deligne cohomology group.

Let $N \geq 3$, $F=\BQ(\mu_N)$, $G=\Gal(F/\BQ)$. Recall that $\Psi_N$ is the $N$-th real cyclotomic polynomial given by the minimal polynomial of $\zeta_N+\zeta_N^{-1}$, where $\zeta_N\in \mu_N$ is a primitive $N$-th root of unity. 
Set $f_N(x_1,x_2):=\Psi_{N}(x_1+x_2) \in \BZ[x_1,x_2]$. Let $X=\Spec \BZ[\mu_N]$, $Z\subset \BG_{m,\BZ}^2$ be the zero locus of $f_N$ and $A$ be the union of the connected components of 
\[\ov{\BC^1\times \{z\in\BC :  |z|<1\}\cap Z(\BC)} \] 
which is a real manifold of dimension $1$. 
Note that $\partial A$ has a natural integral model over $\BZ$ as a subscheme of $Z$, where we denote by $\partial A$ abusing the notation. We have $\partial A_{\BQ}\simeq X_{\BQ}$ and $H_{\CM}^1(X,\BQ(2))=H_{\CM}^1(\partial A,\BQ(2))$.

Let $x \in \BZ[x^\pm]^\times\otimes \BQ\simeq H_{\CM}^1(\BG_{m,\BZ},\BQ(1))$ and $\wt{c}_j\in H_{\CM}^1(Z,\BQ(1))$ be the pull-back of $x$ under the map \[Z\hookrightarrow \BG_{m,\BZ}^2\xrightarrow{p_j}\BG_{m,\BZ}\] where the first map is the inclusion and $p_j$ is the projection for $j \in \{1,2\}$. Then $\wt{c}_j$ arises from a unique element $c_j\in H_{\CM}^1(Z,\partial A,\BQ(1))$.
Set \[c:=c_1\cup c_2\in H_{\CM}^2(Z,\partial A,\BQ(2))^{i=-1} \] where $i$ is an involution induced by the interchange of two coordinates on $\BG_m^2$.

Let $Y=\Spec \BZ[\mu_N]^+$. A natural morphism $X\ra Y$ induces  $H^1(X,\BQ(1))\simeq H^1(Y,\BQ(1))$ and $H^1_{\CD}(X_{\BR},\BR(1))\simeq H^1_{\CD}(Y_{\BR},\BR(1))$. Let $Y'=\Spec \BZ[\mu_N]^+[\frac{1}{\zeta_N+\zeta_N^{-1}}]\subset Y$ when $N\neq 4$, and $Y'=Y$ when $N=4$.

\begin{thmL} \label{thm:reg} 
We have the following.
\begin{enumerate}
\item There exists a commutative diagram with that rows are exact and columns maps are given by the regulator maps which are injective: 
\[\xymatrix{
0\ar[r]& H_{\CM}^1(X,\BQ(2))\ar@{_{(}->}[d]^{r_\CD}\ar[r]& H_{\CM}^2(Z,\partial A,\BQ(2))^{i=-1}\ar@{_{(}->}[d]^{r_\CD}\ar[r]& H_{\CM}^1(Y',\BQ(1))\ar@{_{(}->}[d]^{r_\CD}\ar[r]& 0\\
0\ar[r]& H_{\CD}^1(X_\BR,\BR(2))\ar[r]& H_{\CD}^2((Z,\partial A)_{\BR},\BR(2))^{i=-1}\ar[r]& H_{\CD}^1(Y_\BR,\BR(1))\ar[r]& 0
.} \]

\item Suppose that $N\neq 3,6$. The regulator map
\[r_{\CD}: H_{\CM}^2(Z,\partial A,\BQ(2))\xrightarrow{r_\CD}H_{\CD}^2((Z,\partial A)_{\BR},\BR(2))\xrightarrow{\frac{1}{2\pi i}\int_A} \BC\]
satisfies  
\begin{equation}\label{mainid}m(f_N^*)-m(f_N)=r_\CD(c)=\sum_{\chi\in \wh{G}}r_{N,\chi} L^{(N_\chi)'}(-\epsilon,\chi)\end{equation} where $N_\chi=\begin{cases}N/4,&\quad 4\parallel N, \text{$\chi$ even, $\cond(\chi)\mid N/4$}\\ 
    N,&\quad \text{otherwise}\end{cases}$, $\epsilon \in \{0, 1\}$ with $\chi(-1)=(-1)^{\epsilon}$ and $r_{N,\chi}\in \BQ(\chi)$ is an explicit constant determined in Corollary~\ref{exp}.
\end{enumerate}
\end{thmL}

\begin{remark}
According to Beilinson, the motivic cohomology groups in the first column (resp. the third column) have a well-understood $G$-module structure whose $\chi$-part for odd character $\chi$ (resp. even) is related to $L'(-1,\chi)$ (resp. $L'(0,\chi)$). We remark that our results including the next theorem do not depend on Beilinson's result, except for an explicit characterisation of the image of regulator maps in the first and third columns in terms of the derivative of $L$-values (cf.~Corollary~\ref{imreg}).

The part (2) on the connection between $m(f_N)$ and twisted sum of Dirichlet $L$-values is basically in line with Symth's work (cf.~Lemma~\ref{lm}). Here, we further stress an intrinsic reason why one should expect \eqref{mainid} holds. Note that our choice of $(Z,\partial A)$ is closely related to $X:=\BP^1_{\BZ[\mu_N]}$, whose regulator of the second motivic cohomology with $\BQ(2)$-coefficients is associated to the leading coefficient of $\zeta(s,X)=\prod_{\chi\in \wh{(\BZ/N\BZ)^\times}}L(s,\chi)L(s-1,\chi)$ at $s=0$ and $2$, as predicted by Beilinson--Bloch--Kato conjecture.
\end{remark} 
 
\begin{remark} 
Note that for a Dirichlet character $\chi$, to consider its $L$-value, it is enough to assume that $\cond(\chi)$ is odd or $4 \mid \cond(\chi)$. For an odd character $\chi$, we have \[\begin{cases}
r_{\cond(\chi),\chi},&\quad \text{$\cond(\chi)$ is odd} \\
r_{2\cond(\chi),\chi} ,&\quad 4|\cond(\chi)
\end{cases}\] is a simple non-zero rational multiple of $$L(0,\chi^{-1}),$$ and hence in particular is non-zero (cf.~Lemma~\ref{nzo}).
\end{remark}


\medskip

Next, we discuss a Galois structure of $H_{\CM}^2(Z,\partial A,\BQ(2))^{i=-1}$, hence the connection between $\chi$-part of the cohomology class and $L'\left(\frac{\chi(-1)-1}{2},\chi\right)$. Note that $\Aut(Z_{\BQ},\partial A_{\BQ})$ is very small so that $H_{\CM}^2(Z,\partial A,\BQ(2))^{i=-1}$ does not admit a direct $G$-action. In this regard, we 
introduce a conjecture on the linear independence of the derivative of partial Dirichlet $L$-values in the mixed sense, which is crucial for us to study the Galois module structure of $H_{\CM}^2(Z,\partial A,\BQ(2))^{i=-1}$.

Let $\CL_{-1}$ and $\CL_{0}$ be the $\BQ$-subspace of $\BR$ generated by the image of regulator pairings:
\begin{equation}\label{1.1}H_{\CM}^1(X,\BQ(2))\times H_{\text{sing},0}(X_\BR,\BQ(-1)) \ra \BR,\quad (x,y)\mapsto \pair{r_{\CD}(x),y},\end{equation}
\begin{equation}\label{1.2}H_{\CM}^1(Y',\BQ(1))\times H_{\text{sing},0}(X_\BR,\BQ) \ra \BR,\quad (x,y)\mapsto \pair{r_{\CD}(x),y}.\end{equation}
In fact, $\CL_{-\epsilon}$ has a precise $\BQ$-basis given by the derivative of partial $L$-values at $s=-\epsilon$. We refer to Corollary \ref{imreg} for an explicit characterisation.

Hence we formalise the hypothesis as follows. Let $\CL$ be the image of regulator pairing
\begin{equation}\label{1.1}H_{\CM}^2(Z,\partial A,\BQ(2))^{i=-1}\times H_{\text{sing},1}((Z,\partial A)_{\BR},\BQ(-1))^{i=-1} \ra \BR.\end{equation}
Then we have $\CL_{-1}+ \CL_0\subset \CL$ by Theorem \ref{thm:reg}.   The following hypothesis asserts the converse inclusion and splitting.
\begin{conj}\label{conj:linind}
For $N\geq 3$, we have
$\CL=\CL_{-1}\oplus \CL_{0}$.
\end{conj}

We refer to \S\ref{subsec:linin} that Conjecture \ref{conj:linind} is essentially related to the linear independence of partial $L$-values in the mixed sense, which naturally seems to be true but out of reach except a few instances (e.g. see \cite{calegari2024}).

Our second result is about the refined identity for a single Dirichlet $L$-value, which follows from the Galois module structure of relevant motivic cohomology.

\begin{thmL} \label{thm:motivic}
Under Conjecture \ref{conj:linind}, 
\begin{itemize}
    \item[(1)] There exist canonical isomorphisms 
    \begin{align*} H_\CM^2(Z,\partial A, \BQ(2))^{i=-1}&\simeq H_{\CM}^1(X,\BQ(2))\oplus H_{\CM}^1(Y',\BQ(1)) \\ 
    H_\text{sing,1}((Z,\partial A)_{\BR}, \BQ(-1))^{i=-1,\circ} &\simeq H_\text{sing,0}(X_{\BR},\BQ(-1)) \oplus H_\text{sing,0}(Y_{\BR},\BQ)^{\circ} \end{align*} that are compatible with regulator pairings on them. Here, we use the upper-index $\circ$ for the quotient of singular homology by the kernel of regulator pairing.
    
\item[(2)]     In particular, there is a canonical $G$-module structure on $H_\CM^2(Z,\partial A, \BQ(2))^{i=-1}$ that is compatible with the regulator map. Hence we have \[ H_\CM^2(Z,\partial A, \BQ(2))^{i=-1}\simeq \begin{cases}\BQ[G]/<\sum_{g\in G}g>,& N\neq 4p^r\\
\BQ[G],&N= 4p^r\end{cases}\] as $G$-modules.  

\item[(3)] For each $\chi\in \wh{G}$, denote by $c^\chi$ the $\chi$-part of $c$. Then we have 
    \begin{equation}\label{chipart}r_\CD(c^\chi)=r_{N,\chi} L^{(N_\chi)'}(-\epsilon,\chi)\end{equation}
    where $\epsilon \in \{0,1\}$ with $\chi(-1)=(-1)^\epsilon$ , $N_\chi | N$ and $r_{N,\chi}\in \BQ(\chi)$ are as in Theorem \ref{thm:reg}.   
\end{itemize} 
\end{thmL}

\begin{remark}
In fact, we can have Theorem \ref{thm:motivic}.(3) under a weaker hypothesis. We consider the image $\CL_{0}^{\sharp}$ of the regulator pairing \eqref{1.2} restricted to $\Im(c)\times H_{\text{sing},0}(X_\BR,\BQ)$ (See Section \ref{weakstm} for an explicit description). Under the natural assumption that $\CL_{-1}\cap \CL_{0}^{\sharp}=\{0\}$, we can still talk about $\chi$-part of $c$ and obtain the refined identity \eqref{chipart}.
\end{remark}

\subsection{Overview of the proof}

We now briefly describe our approach, which  closely follows \cite{Den:97, Bes:Den} along with appropriate adaptions for the setting introduced in Section \ref{ms}. 

The seminal arguments of \cite{Den:97} on the realisation of Mahler measure as an image of the regulator map of the associated motivic cohomology class in the Deligne cohomology, hence, yield a systematic general framework for the study of special values of $L$-functions upon suitable choices of the corresponding geometric objects; See Section \ref{sec:pf:reg} for details. Indeed, in \cite[3.10]{Bes:Den}, they concluded that there is a family of polynomials in $\BQ[x_1,x_2]$ associated with an elliptic curve $E/\BQ$ whose Mahler measure is precisely given by $L'(0,E)$ as an image of the regulator map under CM and technical conditions.  

Here, we determine corresponding objects for the Dirichlet case. We present a family of polynomials $f_N$ in $\BZ[x_1,x_2]$ associated with $N$-th cyclotomic polynomial whose Mahler measure is expressed in terms of a linear combination of special $L$-value of mod $N$ Dirichlet characters at $s=0,-1$ in Theorem \ref{thm:reg}. In this $\GL_1$ situation, we remark that the image of regulator map is not directly given by a single $L$-value due to Galois conjugate structure in the exact sequences.

Theorem \ref{thm:motivic} then shows a core technical step to extract the single Dirichlet $L$-value by taking $\chi$-part of the constructed class attached to $f_N$ in motivic cohomology $H_\CM^2$. To do so, we need to assume Conjecture \ref{conj:linind} on linear independence among partial $L$-values, which enables us to obtain a Galois action at the level of Deligne cohomology. Hence we have an explicit way to produce the elements and confirm the identities that match with Beilinson's conjecture; See Section \ref{remap} and Proposition \ref{partind}.

In the subsequent work, we continue to explore this direction. From the cyclotomic field instance in the present paper, it seems that a refinement of general Beilinson's conjecture on the cyclicity of motivic cohomology under finite abelian extension should hold. Further, there should be an explicit generator in motivic cohomology so that it maps to the partial derivative $L$-values with the Euler factor removed under the Beilinson regulator map. When this stronger conjecture holds, we expect to have the existence of Euler system that could be read from the corresponding regulator map and linear independence among partial $L$-values. Such phenomena seem to match with known examples.

\begin{remark}

There has been a wealth of previous studies on finding the families of polynomials whose Mahler measure can be related to special values of Dirichlet $L$-functions. For instance, we refer to Boyd and Rodriguez-Villegas \cite{Boyd:Rod, Ville} and Guilloux--Marche \cite{guil:mar}, where the proofs are based on the evaluation of Bloch-Wigner dilogarithm, which can be written as linear combination of Dirichlet $L$-values at $s=-1$ for odd characters. These results can be viewed as numerical evidences towards Chinburg's conjecture, while our work follows cohomological methods thanks to Deninger and is towards Iwasawa theoretic applications.  

We end with a brief exposition to the question imposed by Chinburg in his unpublished note \cite{Chinburg}: Let $\chi$ be a quadratic character such that $L(s,\chi)$ has a first-order zero at $s=-n$, where $n \in \BN$. Then there exists a non-zero polynomial $f \in \BZ[x_1^\pm, \cdots, x_{n+1}^\pm]$ such that 
\begin{equation} \label{eq:chin}
m(f)= r_\chi \cdot L'(-n,\chi) \end{equation} for some $r_\chi \in \BQ$. A generalisation of this conjecture to all odd primitive characters can be found in \cite[Conjecture 1.7]{Bertin} with $L'(-n,\chi)$ replaced by its real part. These remain completely open, however, we have a wide collection of supporting progress including the aforementioned works.  

Our Theorem \ref{thm:reg} and \ref{thm:motivic} suggest that we may also obtain a Chinburg-type identity \eqref{eq:chin} with a polynomial $f_N^\chi \in \BZ[x_1^\pm, x_2^\pm]$ corresponding to the $\chi$-part of motivic cohomology class and with $r_\chi \in \BQ(\chi)$ under Conjecture \ref{conj:linind}. However, it seems that the polynomial $f_N^\chi$ and constant $r_\chi$ cannot be taken from $\BZ$-coefficient ring and $\BQ$ respectively as follows. Let $F$ and $K$ be the number fields, and set \[M_n(K,F):=\{k\cdot m(f): k\in K, f\in F[x_1,\cdots,x_n]\}\subset \BC.\] 
For $N \in \mathbb{N}$, let $K_N$ be the number field generated by the image of all Dirichlet characters modulo $N$.
Then it is not hard to see from Jensen's formula and special value formula that for an even $\chi$, we have
\[L'(0,\chi)\in M_1(K_N,\BQ(\mu_N)^+) .\] 
Similarly, for an odd $\chi$, our statement seems to suggest a variant of the conjecture for all primitive characters with a larger coefficient field
\begin{equation}\label{conj}L'(-1,\chi)\in M_2(K_N,\BQ(\mu_N)^+) .
\end{equation}

To support \eqref{conj}, we give the following numerical reasoning. Let $\chi_3$ be a unique primitive odd Dirichlet character modulo $3$,
then it follows from Theorem \ref{thm:reg} that 
\[m(f_3)=L'(-1,\chi_3)\] hence \eqref{conj} holds for $N=3$. Note that $\chi_3$ is quadratic, which is the case that matches with Chinburg's conjecture.

Let $\chi_5$ be an odd Dirichlet character modulo $5$ such that $\chi_5(2)=i$. It then follows that 
\[-m(f_5)\equiv \frac{3-i}{10}L'(-1,\chi_5)+\frac{3+i}{10}L'(-1,\chi_5^3)\pmod{M_1(K_5,\BQ(\mu_5)^+)}\]
\[-m(f_{10})\equiv \frac{1-i}{10}L'(-1,\chi_5)+\frac{1+i}{10}L'(-1,\chi_5^3)\pmod{M_1(K_{5},\BQ(\mu_5)^+)}\]
hence \eqref{conj} holds for $N=5$.

Let $\chi_{15}=\chi_3\chi_5^2$ be the unique primitive odd character modulo $15$, then it follows that
\[-m(f_{15})\equiv\frac{1}{15}L'(-1,\chi_{15})-\frac{3}{5}L'(-1,\chi_3)-\frac{-1+i}{6}L'(-1,\chi_5)-\frac{-1-i}{6}L'(-1,\chi_5^3)\] 
modulo $M_1(K_{15},\BQ(\mu_{15})^+)$, in turn \eqref{conj} holds for $N=15$.
\end{remark}

\begin{remark}
Another aspect of the Mahler measure can be found in the context of dynamical system. For a $\BZ^n$-action $\alpha$ on abelian group $X$ defined by a morphism $\BZ^n \ra \mathrm{Aut}(X)$, Lind--Schmidt--Ward \cite{LSW:90} showed that the topological entropy $h(\alpha)$ is given by the Mahler measure of associated polynomial with the integral coefficients. In particular, for $f \in \BZ[x_1^\pm, \cdots, x_n^\pm]$ and  $X=\BZ[x_1^\pm, \cdots, x_n^\pm] / (f)$, they have $h(\alpha)=m(f)$. 

A priori, there is no clue for what kind of algebraic properties the topological entropy admits in general, as it arises from metric structures of the dynamical system. However, it seems like a natural and interesting question in relation with fractal dimension theory. Our Theorem \ref{thm:reg} and \ref{thm:motivic} give an algebraic characterisation for the entropy of the associated $\BZ^n$-action in cohomological terms with a closed form in terms of Dirichlet $L$-values, which are conjectured to be  transcendental. Based on this, we plan to study arithmetic properties of the dynamical invariants in the future. 
\end{remark}

\subsubsection*{Organisation} 

In Section \ref{sec:pf:reg}, we recall some preliminaries about Deligne and motivic cohomology, regulator map, work of Deninger and prove Theorem \ref{thm:reg}. In Section \ref{sec:pf:motivic}, we obtain Theorem \ref{thm:motivic}. We also present some explicit calculations of the partial $L$-values and Mahler measure of cyclotomic polynomials in the Appendix.

\subsubsection*{Acknowledgements}

This collaboration started during the TSIMF Workshop: Special values of $L$-functions, took place in Sanya, January 2024. We are grateful to the organisers for the opportunity to spend an intensive week. We thank Gunther Cornelissen for helpful chats, in particular introducing us to the Chinburg conjecture. We also thank anonymous referee for insightful comments that widen our overall understanding of accompanying contexts.

 \section{Mahler measure and regulators} \label{sec:pf:reg}

In this section, we recall some basics of the motivic and Deligne cohomology, regulator maps, and explain the work of Deninger. We then present the proof of Theorem \ref{thm:reg}.

\subsection{Cohomology groups} \label{sec2:prelim}

We closely follow the notations and conventions in Beilinson \cite[Chap.~1]{beilinson1984higher} and Nekovar \cite[Sec.~7]{nekovar1994beilinson} in the following.

\subsubsection{Motivic cohomology}\label{motco}

Let $X$ be a regulor scheme. We recall that the motivic cohomology \[H_{\CM}^j(X,\BQ(i))\] is defined by \[(K_{2i-j}(X)\otimes\BQ)^{(i)},\] where $K_n(X)$ is the $n$-th $K$-group of $X$ introduced by Quillen and $(K_n(X)\otimes_{\BZ} \BQ)^{(i)}\subset K_n(X)\otimes_{\BZ} \BQ $ denotes the subspace on which the Adams operators $\Psi^m$, $m\in \BZ_{>0}$, act by $m^i$.

The product structure on $K$-group induces a cup product:
\[\cup: H_{\CM}^i(X,\BQ(j))\times H_{\CM}^{i'}(X,\BQ(j'))\ra H_{\CM}^{i+i'}(X,\BQ(j+j')) \]
Similarly, one can define the relative $K$-group, relative motivic cohomology and long exact sequence of relative motivic cohomologies.

\subsubsection{Deligne Cohomology}

Let $X$ be a complex algebraic manifold. The Deligne cohomology $H_{\CD}^{\tiny{\bullet}}(X,\BR(n)_{\CD})$ 
is defined as a hypercohomology group of suitable complex of sheaves of abelian groups  $\BR(n)_{\CD}$. If $X$ is proper, then we simply have
 \[\BR(n)_{\CD}=[\BR(n)\ra\CO_X\ra\cdots\ra\Omega_{X}^{n-1}],\] where $\Omega_X^i$ is the sheaf of holomorphic differential $i$-forms on $X$.
 
We recall that the Deligne cohomology fits into the following exact sequences:
\[\cdots\ra H_{\text{sing}}^i(X,\BR(n))\ra H_{\dR}^i(X)/F^n\ra H_{\CD}^{i+1}(X,\BR(n))\ra H_{\text{sing}}^{i+1}(X,\BR(n))\ra\cdots. \]
In particular, for $n>i+1$, \begin{equation} \label{eq:de:sing}
H_{\CD}^{i+1}(X,\BR(n))\simeq H_{\text{sing}}^i(X,\BC/\BR(n)).\end{equation}

By Hironaka's theorem on the resolution of singularities, there exists a compactification $j: X\ra \ov{X}$ of $X$ such that $D:=\ov{X}\bs X$ is a divisor with normal crossings.
Then the Deligne cohomology admits an explicit description: 
\[H_{\CD}^n(X,\BR(n))=\frac{\{\varphi\in A^{n-1}(X,\BR(n-1))\ |\ d\varphi=\pi_{n-1}(\omega), \omega\in \Omega_{\ov{X}}^n(\log D)\}}{dA^{n-2}(X,\BR(n-1))},\]where $A^{i}(X,\BR(j))$ is the $i$-forms on $X$ with coefficients in $\BR(j)$, $\pi_{i-1}:\BC\ra \BR(i-1)$ is the projection along $\BR(i)$ and $\Omega_{\ov{X}}^i(\log D)$ is the sheaf of meromorphic differential $i$-forms on $\ov{X}$ at most with logarithmic singularity along $D$. 

Under this identification, a cup product \[\cup: H_{\CD}^i(X,\BR(i))\times H_{\CD}^j(X,\BR(j))\ra H_{\CD}^{i+j}(X,\BR(i+j)) \] is explicitly given by 
\[\varphi_i\cup \varphi_j=\varphi_i\wedge \pi_j \omega_{\varphi_j}+(-1)^{i}\pi_i \omega_{\varphi_i}\wedge\varphi_j,\]where $\pi_{j-1} (\omega_{\varphi_j})=d\varphi_j$.

If $X$ is a real algebraic manifold, we define the Deligne cohomology with coefficients in $\BR(n)$ by
\[H^i_{\CD}(X_{\BR},\BR(n)):=H_{\CD}^i(X_{\BC},\BR(n))^+,\]where $+$ means that we take the fixed elements under $F_\infty^*\circ (\text{complex conjugation})$, where $F_\infty$ is the anti-holomorphic involution on $(\ov{X}_{\BC},D_\BC)$. One could also consider the relative cohomology in a similar manner.

\begin{example}
Let $F$ be a number field and $X=\Spec F$ viewed as a $\BQ$-scheme. Then for $n>1$, we have 
    \[H_{\CD}^1(X_{\BR},\BR(n))=\left(\prod_{\sigma:F\ra \BC}\BR(n-1)\right)^+=\begin{cases}
        \BR(n-1)^{r_2},&\quad \text{if $n$ is even},\\
        \BR(n-1)^{r_1+r_2},&\quad \text{if $n$ is odd},
    \end{cases}\]
 where $r_1$ (resp. $r_2$) is the number of real embeddings (resp. pairs of complex embeddings).
 \end{example}

\subsubsection{Regulator map}\label{remap}
Let $X$ be a regular scheme over $\BR$. Beilinson introduced a regulator map 
\[r_{\CD}: H_{\CM}^i(X_{\BR},\BQ(j))\ra H_{\CD}^i(X_{\BR},\BQ(j))\]  which induced by the Chern characters. Note that the regulator map is compatible with cup product by the multiplicativity of Chern characters, and also compatible with a long exact sequence of relative cohomology groups. 

Suppose now that $X$ is defined over $\BZ$ or $\BQ$, then one defines the regulator map by the natural composition:
\[ H_{\CM}^i(X,\BQ(j))\ra H_{\CM}^i(X_{\BR},\BQ(j))\xrightarrow{r_\CD} H_{\CD}^i(X_{\BR},\BQ(j))\] and we may still denote it by the notation $r_{\CD}$.  

We remark that there are two rational structures of $H_{\CD}^i(X_{\BR},\BR(j))$. One comes from $H_{\CM}^i(X,\BQ(j))$ under $r_D$, and the other comes from the singular cohomology with rational coefficients. The Beilinson conjecture predicts that 
 \[\ord_{s=i-j}L(s, X)=\dim_{\BR}{H_{\CD}^{i}(X_{\BR},\BR(j))}\]
 and the determinant of the above two rational structures is essentially equal to the leading coefficients of $L(s, X_{})$ at $m=i-j$, whenever $m<\frac{i-1}{2}$ is critical in the sense of Deligne. There are generalisations of the conjecture for special $L$-values at other integers.

\subsubsection{Work of Deninger}

We are now ready to summarise the work of Deninger \cite{Den:97} on the interpretation of the Mahler measure as an image of the regulator map of the corresponding motivic cohomology class in the Deligne coholomogy.

Let $K=\BR$ or $\BC$ and let $0 \neq f \in K[x_1, \cdots, x_n]$ be a polynomial. Let $Z$ be the zero locus of $f$ in $\BG_{m,\BC}^n$. Set $B=\{z\in \BC\ |\ |z|\leq 1\}$ and $T^r:=(S^1)^r\subset \BC^r$ be the real $r$-torus.
\begin{thm}\cite[Thm.~3.4]{Den:97}\label{Den1}
Suppose that the following conditions hold:
\begin{itemize}
    \item [(i)] The divisor of $f$ in $\BG_m^{n-1} \times \BA^1$ over $\BC$ has no multiple components, and $f^*(x_1, \ldots, x_{n-1})\\ := f(x_1, \ldots, x_{n-1}, 0)\in \BC[x_1,\cdots,x_{n-1}]$ does not vanish on $T^{n-1}$.
    \item [(ii)]The set $A$, consisting of $(n-1)$-dimensional connected components of $\ov{(T^{n-1} \times B^\circ) \cap Z}$, is a compact real submanifold of $ Z^{\mathrm{reg}}$ with a boundary.
\end{itemize}
Then,
\begin{itemize}
    \item [(1)]Under the natural pairing
\[
\langle \cdot, \cdot \rangle : H^n_{\CD}(Z^{\mathrm{reg}}_K, \BR(n)) \times H_{\text{sing}, n}(Z^{\mathrm{reg}}_K, \BR(-n)) \to \BR,
\]  
we have
\[
m(f^*) - m(f) = \left\langle [\varepsilon_1] \cup \cdots \cup [\varepsilon_n],\ [A] \otimes (2\pi i)^{1-n} \right\rangle,
\]  
where \( \varepsilon_i = \log |z_i| \in H^1_{\CD}(Z^{\mathrm{reg}}_K, \BR(1)) \).
\item [(2)] Further if $f \in \BQ[x_1, \ldots, x_n]$, then we set $\BQ$-model $Z_{\BQ}$ of $Z$.  Then the cup product \[\{x_1, \ldots, x_n\}:=x_1\cup\cdots\cup x_n\in H^n_{\CM}(Z_{\BQ}^{\mathrm{reg}}, \BQ(n))\] maps under the regulator to the cup product class $[\varepsilon_1] \cup \cdots \cup [\varepsilon_n]$, and we obtain
\[
m(f^*) - m(f) = \left\langle r_{\CD}\{x_1, \ldots, x_n\},\ [A] \otimes (2\pi i)^{1-n} \right\rangle.
\]  

\end{itemize}

\end{thm}

\begin{thm}\cite[Prop.~3.6]{Den:97}\label{Den2}
Let \( 0 \neq f \in \BQ[x_1, x_2] \) satisfy the assumptions of Theorem \ref{Den1} with \( n = 2 \), and suppose that the boundary points \( \partial A \subset T^2 \cap Z_\BC \) are roots of unity. Let $Z_{\BQ}$ and $A_{\BQ}$ be the $\BQ$ model of $Z$ and $\partial A$. Then the symbol \( \{x_1, x_2\} \) lies in the relative motivic cohomology group
\[
H^2_{\CM}(Z_{\BQ}^{\mathrm{reg}}, \partial A_{\BQ}; \BQ(2)),
\]  
and under the regulator map we have
\[
m(f^*) - m(f) = \left\langle r_{\CD}\{x_1, x_2\},\ [A] \otimes (2\pi i)^{-1} \right\rangle,
\]  
where the pairing is taken in the relative Deligne cohomology:
\[
H^2_{\CD}((Z^{\mathrm{reg}}, \partial A)_{\BR}, \BR(2)) \times H_1((Z^{\mathrm{reg}}, \partial A)_{\BR}, \BZ(-1)) \to \BR.
\]
\end{thm}

\subsection{Proof of Theorem \ref{thm:reg}}

Along with preliminaries we recall in Section \ref{sec2:prelim}, we present the proof of our first result Theorem \ref{thm:reg}. The main point to is to make an appropriate choice of the polynomial that captures the Dirichlet $L$-values under the motivic realisations.

\subsubsection{Mahler measure and regulator maps}\label{S:randm}
Let $N \geq 3$ be an integer, and $f:=f_N$, $Z$, $A$, $c$ be as in Theorem \ref{thm:reg}.

\begin{prop}  \label{prop:lincomb}
Suppose that $N\neq 3,6$. We have the regulator map
\[r_{\CD}: H_{\CM}^2(Z,\partial A,\BQ(2))\xrightarrow{r_\CD} H_{\CD}^2((Z,\partial A)_{\BR},\BR(2))\xrightarrow{\frac{1}{2\pi i}\int_{A}} \BC\]
which satisfies  
\begin{align*} 
r_\CD(c)&=m(f^*)-m(f) \\
&=\sum_{\chi\in \wh{G}}r_{N,\chi} L^{(N_\chi)'}(-\epsilon,\chi) 
\end{align*}
where $\epsilon\in \{0,1\}$ with $\chi(-1)=(-1)^{\epsilon}$, $N_\chi\mid N$ and $r_{N,\chi}\in \BQ(\chi)$ are as in Corollary \ref{exp}.
\end{prop}

\begin{proof}
The condition on $N$ implies that $f$ satisfies the assumptions of Theorem \ref{Den1}. Further, the coordinates of $\partial A$ are roots of unities, in turn satisfy the assumptions of Theorem \ref{Den2}. Hence the result follows.
\end{proof}

\begin{remark}
Recall that $i\in \Aut(Z,\partial A)$ is an involution induced by interchanging the two coordinates on $\BG_m^2$. 
Thus we have the restriction map \[ H_{\CM}^2(Z,\partial A,\BQ(2))^{i=\epsilon}\xrightarrow{} H_{\CD}^2((Z,\partial A)_{\BR},\BR(2))^{i=\epsilon}\] which is also well-defined and still denoted by the notation $r_\CD$. Since $c_1\cup c_2=-c_2\cup c_1$, we have \[c=c_1\cup c_2\in H_{\CM}^2(Z,\partial A,\BQ(2))^{i=-1}.\] 
\end{remark}

\subsubsection{Exact sequence}

Let $N\geq 3$, and $X$, $Y'$, $G$, $Z$, $A$ be as in Section \ref{ms} unless otherwise stated.

\begin{prop} \label{exactseq}
We have the following commutative diagram that the  rows are exact and column are given by the  regulator maps 
\[\xymatrix{
0\ar[r]& H_{\CM}^1(X,\BQ(2))\ar[d]^{r_\CD}\ar[r]& H_{\CM}^2(Z,\partial A,\BQ(2))^{i=-1}\ar[d]^{r_\CD}\ar[r]& H_{\CM}^1(Y',\BQ(1))\ar[d]^{r_\CD}\ar[r]& 0\\
0\ar[r]& H_{\CD}^1(X_\BR,\BR(2))\ar[r]& H_{\CD}^2((Z,\partial A)_{\BR},\BR(2))^{i=-1}\ar[r]& H_{\CD}^1(Y_\BR,\BR(1))\ar[r]& 0
.} \]
Further, the image of $c$ in $H_{\CM}^1(Y',\BQ(1))\simeq \BZ[\mu_N]^+[\frac{1}{\cos 2\pi /N}]^\times\otimes \BQ$ is given by $\frac{1}{2\cos (2\pi /N)}$.
    \end{prop}
\begin{proof}The proof uses the properties introduced in \S\ref{motco}--\S\ref{remap} and localisation sequence in $K$-theory for regular schemes (cf.~\cite[Vol.~1~Part~II~Sec.~1.4.2]{friedlander2005handbook}).

We first look at the first row.
We have $\partial A\subset Z$ is closed and by a long exact sequence in the relative motivic cohomology, we have the exact sequence:
\[H^1_{\CM}(Z,\BQ(2))\ra H_{\CM}^1(\partial A,\BQ(2))\ra H_{\CM}^2(Z,\partial A,\BQ(2))\ra H_{\CM}^2(Z,\BQ(2))\ra H_{\CM}^2(\partial A,\BQ(2)).\]
In the following, we show that \[\quad H_{\CM}^2(\partial A,\BQ(2))=0 \mbox{ and } H^1_{\CM}(Z,\BQ(2))=0.\]

By the definition, $H_{\CM}^2(\partial A,\BQ(2))$ is a piece of $K_2(\BZ[\mu_N][1/S])\otimes_{\BZ} \BQ$ for some integer $S$. By the localisation sequence in $K$-theory and the fact that $K_2$ of the finite field equals to $0$, we have $K_2(\BZ[\mu_N][1/S])\otimes_{\BZ} \BQ=K_2(\BQ[\mu_N])\otimes_{\BZ} \BQ$. It is known by the result of Garland that $K_2(F)\otimes_{\BZ} \BQ=0$ for any number field $F$, hence $H_{\CM}^2(\partial A,\BQ(2))=0$.

Next we claim that $H_{\CM}^1(Z,\BQ(2))=0$. In fact, recall that \[Z=\bigsqcup_{1\leq k<N/2, (k,N)=1}Z_k,\] where $Z_k=\{(z_1,z_2): \ z_1+z_2=2\cos({2\pi k }/{N}), z_i\neq 0, i=1,2\}$. 
Let \[S=\bigsqcup_{1\leq k<N/2, (k,N)=1}S_k,\] where $S_k=\{(z_1,z_2): \ z_1+z_2=2\cos({2\pi k }/{N})\}$.  Then $(S\bs Z)_{\BQ}\simeq \Spec \BQ(\mu_N)^+\sqcup \Spec \BQ(\mu_N)^+$ over $\BQ$, where $\BQ(\mu_N)^+$ denotes the maximal totally real subfield of $\BQ(\mu_N)$. We have the following localisation long exact sequence of the motivic cohomology: 
\[ H_\CM^1(S,\BQ(2))\ra H_{\CM}^1(Z,\BQ(2))
    \ra H_{\CM}^0(S\bs Z,\BQ(1)) .\]

It follows similarly as above that $H_{\CM}^0(S\bs Z,\BQ(1))=0$ since $K_2(\BZ(\mu_N)^+)\otimes_{\BZ}\BQ=0$.
We have $S\simeq\Spec \BZ[\mu_N]^+[x]$. Since $\Spec\BZ(\mu_N)^+[x]$ and $\Spec\BZ(\mu_N)^+$ are homotopic, they admit the same motivic cohomology group, thus $H_\CM^1(S,\BQ(2))=H_\CM^1(\Spec\BZ(\mu_N)^+,\BQ(2))$, which is a piece of $K_3(\BZ[\mu_N]^+)\otimes_{\BZ}\BQ$. Note that $K_3(\BZ(\mu_N)^+)\otimes \BQ=0$, since $\BQ(\mu_N)^+$ is a totally real field. Finally, we have $H_{\CM}^1(S,\BQ(2))=0$, hence $H_{\CM}^1(Z,\BQ(2))=0$.

Thus the following sequence is exact: 
\begin{equation}\label{e1}0\ra H_{\CM}^1(\partial A,\BQ(2))\ra H_{\CM}^2(Z,\partial A,\BQ(2))\ra H_{\CM}^2(Z,\BQ(2))\ra 0.\end{equation}
We have an involution compatible with the short sequence \eqref{e1}. Consider the eigenpart of \eqref{e1} with that $i$ acts by $-1$, then we have 
\[H_{\CM}^1(\partial A,\BQ(2))^{i=-1}=K_3(\BZ[\mu_N])_{\BZ}\otimes \BQ.\] In fact, $i$ acts on $\Spec \BQ[\mu_N]$ as a non-trivial element in $\Gal(\BQ(\mu_N)/\BQ(\mu_N)^+)$, which acts on $K_3(\BZ[\mu_N])_{\BZ}\otimes \BQ$ by $-1$.

Let $T_i\subset S$ be given by $z_i=0$ for $i \in \{1,2\}$, then $Z=S \bs (T_1\cup T_2)$.
We now show that $H_{\CM}^2(Z,\BQ(2))^{i=-1}=K_1(T_1\bs (T_1\cap T_2))_{\BQ}$. 

Applying a localisation sequence for $(T_1\bs (T_1\cap T_2))\sqcup (T_2\bs (T_1\cap T_2)) \subset S\bs (T_1\cap T_2)$, we obtain
\[\begin{aligned}
      H_\CM^2(S\bs (T_1\cap T_2),\BQ(2)) &\ra H_{\CM}^2(Z,\BQ(2))
    \ra H_{\CM}^1((T_1\bs (T_1\cap T_2))\sqcup (T_{2}\bs (T_1\cap T_2)) ,\BQ(1))\\
    &\ra H^3_\CM(S\bs (T_1\cap T_2),\BQ(2)). 
\end{aligned}\]
One has $ H_\CM^2(S\bs (T_1\cap T_2),\BQ(2))=0$ and $H^3_\CM(S\bs (T_1\cap T_2),\BQ(2))\simeq H^0((T_1\cap T_2),\BQ)$.
Thus we have 
\[H_{\CM}^2(Z,\BQ(2))=\ker(K_1(T_1\bs (T_1\cap T_2)\sqcup (T_2\bs (T_1\cap T_2)))_{\BQ}\xrightarrow{} K_0(T_1\cap T_2) )_{\BQ} .\]Note that action of $i$ acts on motivic cohomologies induced by interchanging the two coordinates, hence $H_{\CM}^2(Z,\BQ(2))^{i=-1}\simeq K_1(T_1\bs (T_1\cap T_2))_{\BQ}=H_{\CM}^1(Y',\BQ(1))$.

Thus we have the exact sequence in the first row of the diagram. For the second row, the proof goes similarly and rather simpler, via the connection to the singular cohomology. As the regulator map is compatible with the long exact sequence of relative motivic and Deligne cohomology, the result follows.

To conclude the remaining part, we observe that the following map is defined by the tame symbol:
\[\begin{aligned}
   H_{\CM}^1(Z_{\BQ},\partial A_{\BQ},\BQ(2))^{i=-1}\cup H_{\CM}^1(Z_{\BQ},\partial A_{\BQ},\BQ(2))^{i=-1}\ra H_{\CM}^2(Z_{\BQ},\partial A_{\BQ},\BQ(2))^{i=-1}\\
   \ra H_{\CM}^1((T_1\bs (T_1\cap T_2))_{\BQ}\sqcup (T_{2}\bs (T_1\cap T_2))_{\BQ} ,\BQ(1)), 
\end{aligned} \]
\[(x,y)\longmapsto  ((-1)^{\ord_{P}(x)\ord_{P}(y)}x^{\ord_{P}(y)}y^{-\ord_{P}(x)},(-1)^{\ord_{Q}(x)\ord_{Q}(y)}x^{\ord_{Q}(y)}y^{-\ord_{Q}(x)})\]
where $P$ and $Q$ correspond to $(T_1\bs (T_1\cap T_2))_{\BQ}$ and $(T_2\bs (T_1\cap T_2))_{\BQ}$, respectively. Finally for $(x,y)=(z_1,z_2)$, the image is given by 
\[ \left(\frac{1}{2\cos (2\pi  /N)},2\cos (2\pi  /N) \right).\] 
\end{proof}

\begin{example}Consider $N=5$. Let $\chi:\Gal(\BQ(\mu_5))\ra \BC^\times$ be an odd character determined by $\chi(\sigma_2)=i$. It follows directly from the Jensen's formula that $m(f_5^*)=\frac{1}{4}L^{(5)'}(0,\chi_5^2)$. 

Further we have
  \[ m(f_5)=\frac{-1}{10}L^{(5)'}(0,\chi^2)+\frac{3-i}{10}L^{(5)'}(-1,\chi)+\frac{3+i}{10}L^{(5)'}(-1,\chi^3). \]
By the functional equation, we have 
\[L^{(5)'}(-1,\chi)=-i\tau(\chi)\cdot5 \cdot \frac{1}{4\pi}L(-2,\chi^{-1}) \] with $\tau(\chi)=-\sqrt{(5-\sqrt{5})/2}+i\sqrt{(5+\sqrt{5})/2}$. See Appendix \ref{subsec:appendix:exp} for the explicit calculation. This is also compatible with \cite[Ex.~8]{Smyth:81}.

\end{example}

\section{Galois structure of relative cohomology} \label{sec:pf:motivic}

In this section, we prove a canonical splitting of the short exact sequence in Theorem \ref{thm:motivic} under linear independence hypothesis on the partial Dirichlet $L$-values in the mixed sense (cf.~Conjecture~\ref{mc}). 

This allows us to study $\chi$-part of the constructed cohomology class as in Theorem \ref{thm:reg}. The main results are Theorem \ref{wmainthm} and its variant Theorem \ref{wmc} with a weaker assumption.

\subsection{Linear independence}

We formalise the simultaneous linear independence hypothesis on partial $L$-values at different integers, which will play a central role in the splitting argument in the proof of Theorem \ref{thm:motivic}.

\subsubsection{Conjecture} \label{subsec:linin}

Recall that we have a long exact sequence (cf.~Proposition~\ref{exactseq}):
\[\xymatrix{
0\ar[r]& H_{\CM}^1(X,\BQ(2))\ar[r]& H_{\CM}^2(Z,\partial A,\BQ(2))^{i=-1}\ar[r]& H_{\CM}^1(Y',\BQ(1))\ar[r]& 0
.} \]
We impose the following assumption for the study of relative motivic cohomology \[H_{\CM}^2(Z,\partial A,\BQ(2))^{i=-1}.\] In fact, we need a hypothesis on linear independence in the mixed sense, in order to make $H_{\CM}^2(Z,\partial A,\BQ(2))^{i=-1}$ into a $G$-module. 
Recall that $\CL_{-1}$ and $\CL_{0}$ are the $\BQ$-subspaces of $\BR$ generated by the image of regulator pairings:
\[H_{\CM}^1(X,\BQ(2))\times H_{\text{sing},0}(X_\BR,\BQ(-1)) \ra \BR,\quad (x,y)\mapsto \pair{r_{\CD}(x),y},\]
\[H_{\CM}^1(Y',\BQ(1))\times H_{\text{sing},0}(X_\BR,\BQ) \ra \BR,\quad (x,y)\mapsto \pair{r_{\CD}(x),y}.\]
\begin{conj}\label{mc}
    Let $N\geq 3$ be an integer. As $\BQ$-subspaces of $\BR$, we have 
 \[\CL_{-1}\cap \CL_{0}=\{0\},\] i.e. $\CL_{-1}$ and $\CL_{0}$ are linearly independent over $\BQ$. 
    \end{conj}

In the following, we give an explicit description of $\CL_{-\epsilon}$ in terms of the derivative of partial Dirichlet $L$-values (cf.~Corollary~\ref{imreg}) which due to Beilinson. Then the Conjecture \ref{mc} is basically the linear independence of partial $L$-values in the mixed sense, which seems to be true but out of reach except a few instances. For example, see \cite{calegari2024}.

Let $N\geq 3$. Let $F=\BQ(\mu_N)$, $G=\Gal(F/\BQ)$ and $c\in G$ be a generator of $\Gal(\BQ(\mu_N)/\BQ(\mu_{N})^+)$. We fix an embedding $\ov{\BQ}\ra \BC$. For $\epsilon\in \{0,1\}$ and $\sigma\in G$, let 
\[L^{(N)'}(-\epsilon,\sigma)=
\begin{cases}
-\log|(1-e^{2\pi i/N})^\sigma(1-e^{-2\pi i/N})^{\sigma}|,\quad &\epsilon=0\\
\frac{-Ni}{8\pi}\sum_{n=1 }^\infty \frac{(e^{2\pi i n/N}-e^{-2\pi i n/N})^{\sigma}}{n^2},\quad &\epsilon=1.\\
\end{cases} \] Then for $1\neq\chi\in \wh{G}$, we have 
\[\sum_{\sigma\in {\wh{G}}}\chi(\sigma)L^{(N)'}(-\epsilon,\sigma)=\begin{cases}
   L^{(N)'}(-\epsilon,\chi) ,&\quad \text{$\chi(-1)=(-1)^\epsilon$}\\
  0  ,&\quad \text{$\chi(-1)=-(-1)^\epsilon$}.\\
\end{cases}\]
For $\epsilon \in \{0,1\}$ and $x=\sum_{\sigma\in G} a_\sigma \sigma \in \BQ[G]$, set
\[ L^{(N)'}(-\epsilon,x):=\sum_{\sigma\in G}a_\sigma L^{(N)'}(-\epsilon,\sigma).\]
Then we have a $\BQ$-linear independence among the derivative of partial $L$-values each of them generated by even and odd characters as follows.
\begin{prop}\label{partind}For $\epsilon \in \{0,1\}$,
the map \[\BQ[G]\ra \BR,\quad x\mapsto L'(-\epsilon,x)\] has a kernel \[\begin{cases}
<1-c,\sum_{g\in G}g>,&\quad \epsilon=0,\\
<1+c>,&\quad \epsilon=1.\\
\end{cases}\]
\end{prop}
\begin{proof}Let $X=\Spec F$.
Let $c\in G$ be the generator of $\Gal(\BQ(\mu_N)/\BQ(\mu_{N})^+)$. We identify $(\BZ/N\BZ)^\times$ with $G$ via \[a\mapsto (\sigma_a: \zeta_N\longmapsto \zeta_N^a).\]
It follows from the Beilinson conjecture (cf. \cite[Thm 6.7]{CT} or \cite[Cor 7.1.6]{beilinson1984higher}) that \begin{equation}\begin{cases}
H_{\CM}^1(X,\BQ(1))\simeq \BQ[G]/<c-1,\sum_{g\in G}g>\\
H_{\CM}^1(X,\BQ(2 ))\simeq \BQ[G]/<c+1>
\end{cases}\end{equation} as $\BQ[G]$-modules. Further, for $\epsilon\in \{0,1\}$, there exist the generators $u_\epsilon\in 
H_{\CM}^1(X,\BQ(\epsilon+1))$ and  $v_\epsilon\in 
H_{\text{sing}, 0}(X_{\BR},\BQ(\epsilon))$ such that 
\[\pair{r_{\CD}(\sigma_a(u_\epsilon)), \sigma_b(v_\epsilon)}\in \BQ^\times \cdot 
{L^{(N)}}'(-\epsilon,\sigma_{ab})\]where $r_\CD$ is the regulator map, $\pair{\ ,\ }$ denotes the natural pairing between Deligne cohomology (singular cohomology in this case) and singular homology,
 respectively. 
 Then the determinant of the regulator map 
\[ H_{\CM}^1(X,\BQ(\epsilon+1))\xrightarrow{r_{\CD}}H_{\CD}^1(X_{\BR},\BQ(\epsilon+1))\] is up to a non-zero rational constant equal to 
\[\det\begin{pmatrix}
\pair{r_\CD(i(u_\epsilon)),j(v_\epsilon)}
\end{pmatrix}_{i, j\in S_\epsilon} \]
where $S_\epsilon$ denotes a basis of  \[\begin{cases}
\BQ[G]/<1-c,\sum_{g\in G}g>,&\quad \epsilon=0,\\
\BQ[G]/<1+c>,&\quad \epsilon=1.\\
\end{cases}\]
The result then directly follows from the non-zeroness of the regulators. 
 \end{proof}

In the following we will use the notation $N=4p^r$ to mean $N$ is of the form $4p^r$ for some prime $p$ and integer $r\geq 1$. Similar for $N\neq 4p^r$.
\begin{cor}\label{imreg}
Let $\zeta_N$ be a primitive $N$-th root of unity. The following maps are well-defined and are isomorphisms of $\BQ$-vector spaces:
\begin{enumerate}
 \item We have \[\begin{aligned}
\BQ[G]/<1+c>\simeq& \CL_{-1}\\
x\mapsto &L^{(N)'}(-1,x).
 \end{aligned}\]
    \item  We have \[\begin{aligned}
    \begin{rcases}
    \BQ[G]/<1-c,\sum_{g\in G}g>,& N\neq 4p^r \\ 
 \BQ[G]/<1-c,\sum_{g\in G}g>\oplus \BQ e, &N=4p^r
    \end{rcases}\simeq & \CL_0,\\
    x\mapsto& L^{(N)'}(0,x),\\
    e\mapsto& \log|\zeta_N+\ov{\zeta}_N|,
 \end{aligned}\]
\end{enumerate}

\end{cor}
\begin{proof}
For (1), the result follows from Proposition \ref{partind} and its proof is due to Beilinson. For (2), the case is similar together with Dirichlet unit theorem. Here we use the fact that $\zeta_N+\ov{\zeta}_N$ is not a global unit if and only if $N=4p^r$ for some $r \geq 1$ and a prime $p$. 
\end{proof}

\subsection{Proof of Theorem \ref{thm:motivic}}\label{S:thmBC} 

We get back to our commutative diagram with rows induced by the long exact sequence of relative cohomology/homology groups, where column maps given by the regulator maps (cf.~Proposition~\ref{exactseq}):
\begin{equation}\label{cmm}\xymatrix{
0\ar[r]& H_{\CM}^1(X,\BQ(2))\ar[d]^{r_\CD}\ar[r]& H_{\CM}^2(Z,\partial A,\BQ(2))^{i=-1}\ar[d]^{r_\CD}\ar[r]& H_{\CM}^1(Y',\BQ(1))\ar[d]^{r_\CD}\ar[r]& 0\\
0\ar[r]& H_{\CD}^1(X_\BR,\BR(2))\ar@{}[d]^\times\ar[r]& H_{\CD}^2((Z,\partial A)_{\BR},\BR(2))^{i=-1}\ar[r]\ar@{}[d]^\times& H_{\CD}^1(X_\BR,\BR(1))\ar@{}[d]^\times\ar[r]& 0\\
0& H_{\text{sing},0}(X_\BR,\BQ(-1))\ar[l]\ar[d]& H_{\text{sing},1}((Z,\partial A)_{\BR},\BQ(-1))^{i=-1}\ar[l]\ar[d]& H_{\text{sing},0}(X_\BR,\BQ)\ar[l]\ar[d]&\ar[l] 0\\
&\BR&\BR&\BR&
.} \end{equation}

Recall that we have $\CL_{-1}, \CL$, and $\CL_{0}$ given by the image of regulator pairings on the first, second and third column respectively. Note that by Proposition~\ref{partind}, we have $\CL_{-1} + \CL_{0} \subset \CL$ and by  Conjecture \ref{mc}, we have $\CL_{-1} \cap \CL_{0}=\{0\}$. Along with natural assumptions on linear independence, we further expect that a stronger property holds. 

\begin{conj} \label{mcs}
We have a canonical splitting
\[ \CL=\CL_{-1} \oplus \CL_{0}. \]
\end{conj}

Finally, we are ready to prove Theorem \ref{thm:motivic}. We remark that the argument uses a simple auxiliary lemma on the split exact sequence, which we will present in Appendix \ref{subsec:auxiliary}. In the following, we use the upper-index $\circ$ to denote the quotient of singular homology by the kernel of regulator pairings on each column of the diagram \eqref{cmm}.
\begin{thm} \label{wmainthm}
Under Conjecture \ref{mcs}, 
\begin{itemize}
    \item[(1)] There exist canonical isomorphisms 
    \begin{align*} H_\CM^2(Z,\partial A, \BQ(2))^{i=-1}&\simeq H_{\CM}^1(X,\BQ(2))\oplus H_{\CM}^1(Y',\BQ(1)) \\ 
    H_\text{sing,1}((Z,\partial A)_{\BR}, \BQ(-1))^{i=-1,\circ} &\simeq H_\text{sing,0}(X_{\BR},\BQ(-1)) \oplus H_\text{sing,0}(Y_{\BR},\BQ)^{\circ} \end{align*} that are compatible with the regulator pairings on them. 
    
    In particular, there is a canonical $G$-module structure on $H_\CM^2(Z,\partial A, \BQ(2))^{i=-1}$ and $H_\CD^2((Z,\partial A)_{\BR}, \BQ(2))^{i=-1}$ compatible with the regulator maps.
    \item[(2)] For each $\chi\in \wh{G}$, denote by $c^\chi$ the $\chi$-part of $c$. Then we have 
    \[r_\CD(c^\chi)=r_{N,\chi} \cdot {L^{(N_\chi)}}'(-\epsilon,\chi)\] 
    where $\epsilon\in \{0,1\}$ with $\chi(-1)=(-1)^\epsilon$, $N_\chi\mid N$, $r_{N,\chi}\in \BQ(\chi)$ are as in Corollary \ref{exp}.
\end{itemize} 
\end{thm}
\begin{proof}

We have a similar diagram with that rows are exact:
\begin{equation}\label{me}\xymatrix{
0\ar[r]& H_{\CM}^1(X,\BQ(2))\ar[d]^{r_\CD}\ar[r]& H_{\CM}^2(Z,\partial A,\BQ(2))^{i=-1}\ar[d]^{r_\CD}\ar[r]& H_{\CM}^1(Y',\BQ(1))\ar[d]^{r_\CD}\ar[r]& 0\\
0\ar[r]& H_{\CD}^1(X_\BR,\BR(2))\ar@{}[d]^\times\ar[r]& H_{\CD}^2((Z,\partial A)_{\BR},\BR(2))^{i=-1}\ar[r]\ar@{}[d]^\times& H_{\CD}^1(X_\BR,\BR(1))\ar@{}[d]^\times\ar[r]& 0\\
0& H_{\text{sing},0}(X_\BR,\BQ(-1))^\circ\ar[l]\ar[d]& H_{\text{sing},1}((Z,\partial A)_{\BR},\BQ(-1))^{i=-1,\circ}\ar[l]\ar[d]& H_{\text{sing},0}(X_\BR,\BQ)^\circ\ar[l]\ar[d]&\ar[l] 0\\
&\BR&\BR&\BR&
.} \end{equation}

Note that we have $H_{\text{sing},0}(X_{\BR},\BQ(-1))^\circ=H_{\text{sing},0}(X_{\BR},\BQ(-1))$ and 
\[ H_{\text{sing},0}(X_\BR,\BQ)^\circ=\begin{cases}
   H_{\text{sing},0}(X_\BR,\BQ),\quad &N=4p^r,\\
H_{\text{sing},0}(X_\BR,\BQ)/\BQ \sum_{\sigma: \BQ(\mu_N)^+\ra \BR}\sigma,\quad & \text{otherwise}.
\end{cases}\]

Then the item (1) is straightforward by comparing the diagram \eqref{me} to that in Lemma \ref{li}, under Conjecture \ref{mcs}. The item (2) follows from (1) along with a priori discussion from \S\ref{S:randm}.
\end{proof}

\subsection{Weak form of Theorem \ref{thm:motivic}} \label{weakstm}

We shortly remark that indeed $\sharp$-part of Theorem \ref{wmainthm} holds under a weaker assumption than Conjecture \ref{mc}, stated in Conjecture \ref{wmc} below.

Let $H_{\CM}^2(Z,\partial A,\BQ(2))^{i=-1,\sharp}\subset H_{\CM}^2(Z,\partial A,\BQ(2))^{i=-1}$ be a $\BQ$-subspace generated by the image of $H_{\CM}^1(X,\BQ(2))$ and $c$. Let 
$H_{\CM}^1(Y',\BQ(1))^{\sharp}\subset H_{\CM}^1(Y',\BQ(1))$ be a $\BQ$-subspace given by the image of $H_{\CM}^2(Z,\partial A,\BQ(2))^{i=-1,\sharp}$.
Let $H_{\text{sing},1}((Z,\partial A)_{\BR},\BQ(-1))^{i=-1,\sharp}$ and $H_{\text{sing},0}(X_\BR,\BQ)^{\sharp }$ be the quotient of related singular homology groups by the kernel of regulator pairing restricted to $\sharp$-part of the motivic cohomology on each column in diagram \eqref{cmm}, respectively. 

We have the following commutative diagram with that rows are exact: 
\begin{equation}\label{me1}\xymatrix{
0\ar[r]& H_{\CM}^1(X,\BQ(2))\ar[d]^{r_\CD}\ar[r]& H_{\CM}^2(Z,\partial A,\BQ(2))^{i=-1,\sharp}\ar[d]^{r_\CD}\ar[r]& H_{\CM}^1(Y',\BQ(1))^{\sharp}\ar[d]^{r_\CD}\ar[r]& 0\\
0\ar[r]& H_{\CD}^1(X_\BR,\BR(2))\ar@{}[d]^\times\ar[r]& H_{\CD}^2((Z,\partial A)_\BR,\BR(2))^{i=-1}\ar[r]\ar@{}[d]^\times& H_{\CD}^1(X_\BR,\BR(1))\ar@{}[d]^\times\ar[r]& 0\\
0& H_{\text{sing},0}(X_\BR,\BQ(-1))\ar[l]\ar[d]& H_{\text{sing},1}((Z,\partial A)_{\BR},\BQ(-1))^{i=-1,\sharp}\ar[l]\ar[d]& H_{\text{sing},0}(X_\BR,\BQ)^\sharp\ar[l]\ar[d]&\ar[l] 0\\
&\BR&\BR&\BR&
.} \end{equation}
Define $\CL_{-1}^{\sharp}, \CL^\sharp$ and $\CL_0^\sharp$ be the image of regulator pairings on the first, second and third column of \eqref{me1} respectively.
Then we have $\CL_{-1}=\CL_{-1}^{\sharp}$. It follows from the second part of Proposition \ref{exactseq} that $\CL_{0}^\sharp\subset \CL_{0}$ is given by 
    \[\CL_{0}^\sharp=\sum_{\zeta_N\in \mu_N^{\circ}}\BQ \cdot \log|\zeta_N+\overline{\zeta}_N|,\] where $\mu_N^\circ$ denotes the set of primitive $N$-th roots of unity. 
    
The following is then an immediate consequence of Conjecture \ref{mc}.

\begin{conj}\label{wmc}
    Let $N\geq 3$ be an integer. As $\BQ$-subspaces of $\BR$, we have 
 \[ \CL_{-1} \cap \CL_{0}^\sharp=\{0\},\] i.e. $\CL_{-1}$ and $\CL_{0}^\sharp$ are linearly independent over $\BQ$. 
 \end{conj}

\begin{thm} \label{thm:wmc}
Under Conjecture \ref{wmc}, 
\begin{itemize}
    \item[(1)] There exist canonical isomorphisms 
    \begin{align*} H_\CM^2(Z,\partial A, \BQ(2))^{i=-1,\sharp}&\simeq H_{\CM}^1(X,\BQ(2))\oplus H_{\CM}^1(Y',\BQ(1))^{i=-1,\sharp} \\ 
    H_\text{sing,1}((Z,\partial A)_{\BR}, \BQ(-1))^{i=-1,\sharp} &\simeq H_\text{sing,0}(X_{\BR},\BQ(-1)) \oplus H_\text{sing,0}(Y_{\BR},\BQ)^{i=-1,\sharp} \end{align*} that are compatible with the regulator pairings on them. 
    
    In particular, there is a canonical {$G$-module} structure on $H_\CM^2(Z,\partial A, \BQ(2))^{i=-1,\sharp }$ and $H_\text{sing,1}((Z,\partial A)_{\BR}, \BQ(-1))^{i=-1,\sharp}$ compatible with the regulator maps.
    \item[(2)] For each $\chi\in \wh{G}$, denote by $c^\chi$ the $\chi$-part of $c$. Then we have 
    \[r_\CD(c^\chi)=r_{N,\chi} \cdot L'(-\epsilon,\chi)\] 
    where $\epsilon\in \{0,1\}$ with $\chi(-1)=(-1)^\epsilon$, $N_\chi\mid N$, $r_{N,\chi}\in \BQ(\chi)$ are as in Corollary \ref{exp}. 
\end{itemize} 
\end{thm}

\begin{proof}
For this $\sharp$-part, a crucial point is that we have 
\begin{equation} \label{id:weaker}
\CL^\sharp= \CL_{-1}+\CL_{0}^\sharp , 
\end{equation}
by the explicit regulator consideration. We postpone the detailed exposition to this identity to Appendix \ref{subsec:euler}. Then this allows us to conclude the claim under Conjecture \ref{wmc} by applying Lemma \ref{li} to the above commutative diagram \eqref{me1}
\end{proof}

\appendix

\section{Auxiliary splitting lemma}  \label{subsec:auxiliary}

In this section, we settle an auxiliary splitting lemma in a general setting, which is designed for the study of commutative diagrams \eqref{me} and \eqref{me1}.
\smallskip

\textbf{Assumptions.} Moving forward, we will assume that we have the following commutative diagram of $\BQ$-linear maps
\[\xymatrix{
    0\ar[r]&\BQ^{m}\ar[d]^{i}\ar[r]^{\alpha'}&\BQ^{m+n'}\ar[d]^{j}\ar[r]^{\beta'}&\ar[r]\BQ^{n'}\ar[d]^{k}\ar[r]& 0\\
    0\ar[r]&\BR^m\ar[r]^{\alpha_{\BR}}&\BR^{m+n}\ar[r]^{\beta_{\BR}}&\ar[r]\BR^n\ar[r]& 0\\
}\]
satisfying: 
\begin{itemize}
    \item [(a)]The rows are exact as $\BQ$ and $\BR$-vector spaces, respectively. Further, the second row has a rational structure, which is induced from $\BQ$-linear maps $\alpha, \beta$. 
    \item [(b)]The map between two rows induces a $\BR$-injective map from the first row tensor by $\BR$ to the second row. 
    \item [(c)]Let $I$ be the image of the natural pairing \[i(\BQ^{m})\times (\BQ^m)^\vee\ra \BR,\]where $(\BQ^m)^\vee\subset (\BR^m)^\vee$. Let $J$ and $K$ be defined similarly for the second column and the third column, respectively. Then $J=I\oplus K$.
\end{itemize}

Now we simply write $\alpha$ for $\alpha_{\BR}$ unless otherwise stated. Consider a commutative diagram: 
    \begin{equation}\label{lid}\xymatrix{
    0\ar[r]&\BQ^{m}\ar[d]^{i}\ar[r]^{\alpha'}&\BQ^{m+n'}\ar[d]^{j}\ar[r]^{\beta'}&\ar[r]\BQ^{n'}\ar[d]^{k}\ar[r]& 0\\
    0\ar[r]& i(\BQ^m)\ar@{}[d]^\times\ar[r]^{\alpha}& j(\BQ^{m+n'})\ar@{}[d]^\times\ar[r]^{\beta}& k(\BQ^{n})\ar@{}[d]^\times\ar[r]&0\\
    0&\ar[l](\BQ^m)^\vee\ar[d]&\ar[l]_{\ov{\alpha^\vee}}(\BQ^{m+n})^\vee/A\ar[d]&\ar[l]_{\ov{\beta^\vee}}(\BQ^n)^\vee/B\ar[d]& \ar[l]0\\
    &I&J&K& \\
}\end{equation} where $A$ and $B$ are kernels of the pairings on each column, and $\ov{\alpha^\vee}$ and $\ov{\beta^\vee}$ are the maps induced by $\alpha^\vee$ and $\beta^\vee$, respectively.
    \begin{lem}\label{li}

        There exist a unique section $\delta'$ of $\alpha'$ and unique section  $\ov{\delta^\vee}$ of $\ov{\alpha^\vee}$ in \eqref{lid}. Equivalently, there exist a unique section $\gamma'$ of $\beta'$ and  unique section $\ov{\gamma^\vee}$ of $\ov{\beta^\vee}$ as $\BQ$-linear maps such that the induced splittings \[\BQ^{m+n'}=\alpha'(\BQ^{m})\oplus \gamma'(\BQ^{n'}) \ \mbox{and } (\BQ^{m+n})^\vee/A=\ov{\delta^\vee}((\BQ^m)^\vee)\oplus\ov{\beta^\vee}((\BQ^n)^\vee/B) \]
        admit the following property:
        
        For any $\alpha'(x)+\gamma'(y)\in \BQ^{m+n'}$ with $x\in \BQ^{m}, y\in \BQ^{n'}$ and $\ov{\delta^\vee}(x^\vee)+\ov{\beta^\vee}(y^\vee)\in (\BQ^{m+n})^{\vee}/A$ with $x^\vee\in (\BQ^m)^{\vee},y^\vee\in (\BQ^n)^\vee/B$,
        we have
        \begin{equation}\label{m}\begin{aligned}
           &\pair{\ov{\delta^\vee}(x^\vee)+\ov{\beta^\vee}(y^\vee),j(\alpha'(x)+\gamma'(y))}_{m+n}
            =&  \pair{x^\vee,i(x)}_m+\pair{y^\vee, k(y)}_n \in J=I\oplus K,
        \end{aligned}\end{equation} where the pairings $\pair{\ ,\ }_k$ are given by the natural pairings $\BR^k\times (\BR^k)^\vee\ra \BR$.
        \end{lem}
        \begin{proof}
          For any $z\in \BQ^{m+n'}$ and  $z^\vee\in (\BQ^{m+n})^\vee$, we have \begin{equation}\label{0}\pair{z^\vee,j(z)}_{m+n}= a+b\end{equation} by (c), for a unique $a\in I$ and $b\in K$.
Let $\delta'$ be determined by 
\begin{equation}\label{1}\pair{z^\vee, j\alpha' \delta'(z)}_m=\pair{\alpha^\vee(z^\vee), i\delta'(z)}_m=a\end{equation} as $z^\vee$ and $z$ vary. Then $\delta'$ is well-defined and uniquely determined by the injectivity of $i$. Taking $z=\alpha'(x)$, we have $b=0$. Comparing \eqref{0} with \eqref{1}, we see 
\[\delta'\alpha'=1.\]
Let $\gamma'$ be an element such that $1=\alpha'\delta'+\gamma'\beta'$. Then $\beta'\gamma'=1$. 

Similarly, we define $\ov{\delta^\vee}$ by
\begin{equation}\label{0.2}\pair{\ov{\delta^\vee}\ov{ \alpha^\vee}(z^\vee),j(z)}_{m+n}=a\end{equation} as $z$ and $z^\vee$ vary. Then $\ov{\delta^\vee}$ is well-defined. For $z=\alpha'(x)$, we have 
\begin{equation}\label{1.0}\pair{\ov{\alpha^\vee}\ov{\delta^\vee}\ov{ \alpha^\vee}(z^\vee),i(x)}_{m+n}=a.
\end{equation}
Comparing \eqref{0} with \eqref{1.0} for $z=\alpha'(x)$, we have 
\[\ov{\alpha^\vee}\ov{\delta^\vee}=1.\]
Let $\ov{\gamma^\vee}$ be an element such that $1=\ov{\delta^\vee}\ov{\alpha^\vee}+\ov{\beta^\vee}\ov{\gamma^\vee}$. Then $\ov{\gamma^\vee}\ov{\beta^\vee}=1$. 
We obtain 
\begin{equation}\label{1.00}\pair{z^\vee, j\gamma'\beta'(z)}_m=b  \ \ \mbox{and } \
\pair{\ov{\beta^\vee}\ov{\gamma^\vee}(z^\vee),j(z)}_{m+n}=b.
\end{equation}
 By writing $z=\alpha'(x)+\gamma'(y)$ and $z^\vee=\ov{\beta^\vee}(y^\vee)+\ov{\delta^\vee}(x^\vee)$, now \eqref{0} becomes 
    \[\begin{aligned}
        a+b=&\pair{z^\vee,j(z)}_{m+n}\\
        =&\pair{\ov{\beta^\vee}(y^\vee)+\ov{\delta^\vee}(x^\vee),j(\alpha'(x)+\gamma'(y))}_{m+n}\\
        =&\pair{\ov{\beta^\vee}(y^\vee),j\alpha'(x)}_{m+n}+\pair{\ov{\beta^\vee}(y^\vee),j\gamma'(y)}_{m+n}+\pair{\ov{\delta^\vee}(x^\vee),j\alpha'(x)}_{m+n}+\pair{\ov{\delta^\vee}(x^\vee),j\gamma'(y)}_{m+n}.
    \end{aligned}\]
    
 By definition, $\pair{\ov{\beta^\vee}(y^\vee),j\alpha'(x)}_{m+n}$ lies in both $I$ and $K$. From \eqref{0.2} and \eqref{1.00}, we also have that $\pair{\ov{\delta^\vee}(x^\vee),j\gamma'(y)}_{m+n}$ lies in both $I$ and $K$.
    Hence 
    \[\pair{\ov{\beta^\vee}(y^\vee),j\alpha'(x)}_{m+n}=0\] and \[\pair{\ov{\delta^\vee}(x^\vee),j\gamma'(y)}_{m+n}=0.\]
Note that $\pair{\ov{\delta^\vee}(x^\vee),j\alpha'(x)}_{m+n} \in I$ and $\pair{\ov{\beta^\vee}(y^\vee),j\gamma'(y)}_{m+n}\in K$. Since $J=I\oplus K$, we have 
\[\pair{\ov{\delta^\vee}(x^\vee),j\alpha'(x)}_{m+n}=a,\quad \pair{\ov{\beta^\vee}(y^\vee),j\gamma'(y)}_{m+n}=b
.
\]
We have
\[a=\pair{z^\vee, j\alpha' \delta'(z)}_{m+n}=\pair{z^\vee, \alpha i \delta'(z)}_{m+n}=\pair{\ov{\alpha^\vee}( z^\vee),  i (x)}_{m}.\]
and
\[b=\pair{\ov{\beta^\vee}(y^\vee),j\gamma'\beta' (z)}_{m+n}=\pair{\ov{\beta^\vee}(y^\vee),j(z)}_{m+n}=\pair{y^\vee,k(y)}_{n},\] in turn the result follows.

Finally for the uniqueness of $\delta'$ and  $\ov{\delta^\vee}$, we remark that \eqref{1} and \eqref{0.2} are necessary to make \eqref{m} holds, as $\delta'$ and $\ov{\delta^\vee}$ are uniquely determined by the conditions.       
 \end{proof}

\section{Explicit calculation} \label{sec:calc:mahler}

In this section, we present explicit calculations on the partial $L$-values and Mahler measure of a variant of cyclotomic polynomial. We remark that some of them are already known in the literature. 

In particular, we explain how we obtain the identity \eqref{id:weaker} that plays a key role in Theorem \ref{thm:wmc}, and the explicit determination of  coefficient constant $r_{N,\chi} \in \BQ(\chi)$ in the main theorems.

\subsection{Euler system property} \label{subsec:euler}
We begin by observing the following, which is essentially due to Smyth.

\begin{lem}\label{lm}
For $0\leq \theta\leq \pi
    $, we have
    \[\delta_\theta\cdot m(x_1+x_2-2\cos\theta)=(1-\frac{2\theta}{\pi})\log|2\cos\theta|+\frac{1}{\pi}\sum_{n=1}^\infty\frac{(-1)^{n-1}}{n^2}\sin(2n\theta), \]
 where $\delta_\theta=1$ if $\theta \leq \pi/2$ and $-1$ if $\theta>\pi/2$. Here if $\theta=\pi/2$, then $m(x_1+x_2)=0$.   
\end{lem}
 
\begin{proof}
For $0\leq \theta\leq\pi/2$, this is simply identical to \cite[Chap.~3, Lemma 1]{Smyth:81}. 

 For $\frac{\pi}{2}< \theta\leq {\pi}
            $, this can be reduced to the case $0\leq \theta\leq\pi/2$ via 
\begin{align*}
                m(x_1+x_2-2\cos\theta)
                &=m(-x_1-x_2-2\cos(\pi-\theta))\\
                &=m(x_1+x_2-2\cos(\pi-\theta)).
\end{align*}
\end{proof}

   We now recall from \S\ref{subsec:linin} that for $\sigma\in \Gal(\BQ(\mu_N)/\BQ)$, we have the derivative of partial $L$-values given by
     \begin{equation} \label{partialeq}
     {L^{(N)'}}(0,\sigma)=\sum_{n=1 }^\infty \frac{(e^{2\pi i n/N}+e^{-2\pi i n/N})^{\sigma}}{n},\quad {L^{(N)'}}(-1,\sigma)=\frac{-iN }{8\pi  }\sum_{n=1 }^\infty \frac{(e^{2\pi i n/N}-e^{-2\pi i n/N})^{\sigma}}{n^2} \end{equation}
 with the notation emphasizing their dependence on $N \geq 1$.

 \begin{lem}[Euler system property]\label{Eu}
For a prime $p\mid N$ and positive integer $N'=N/p$, we have
    \begin{itemize}
        \item [(1)] $\sum_{\tau\in  \Gal(\BQ(\mu_N)/\BQ(\mu_{N'}))}L^{(N)'}(0,\tau\sigma_k)=\begin{cases}
        L^{(N')'}(0,\sigma_k),&p\mid N'\\
        L^{(N')'}(0,\sigma_k(1-\sigma_p^{-1})),&p\nmid N'\\
        \end{cases}$.
        \item [(2)] $
        \sum_{\tau\in  \Gal(\BQ(\mu_N)/\BQ(\mu_{N'}))}L^{(N)'}(-1,\tau\sigma_k)=\begin{cases}
        L^{(N')'}(-1,\sigma_k),&p\mid N'\\
        L^{(N')'}(-1,\sigma_k(1-p\sigma_p^{-1})),&p\nmid N'\\
        \end{cases}$.
    \end{itemize}
    where $\sigma_k\in \Gal(\BQ(\mu_N)/\BQ)$ is given by $\sigma_k: \zeta_N\longmapsto \zeta_N^k$.
    \end{lem}
    \begin{proof}
We confirm the property by direct calculations using the explicit forms in Eq.\eqref{partialeq}. Let $G_{N,N'}:=\Gal(\BQ(\mu_N)/\BQ(\mu_{N'}))$ and $\varphi_{N,N'}:(\BZ/N\BZ)^\times\ra(\BZ/N'\BZ)^\times$ be a homomorphism. Then we have 
\begin{align*}
        \sum_{\tau\in G_{N,N'}}L^{(N)'}(0,\tau\sigma_k)
        =& \sum_{n=1 }^\infty \frac{\sum_{s\in \ker \varphi_{N,N'}}(e^{2\pi i \tfrac{sk n}{N}}+e^{-2\pi i \tfrac{sk n}{N}})}{n}\\
        =&\displaystyle \begin{cases}
        \sum_{n=1,p|n }^\infty \frac{p (e^{2\pi i \tfrac{k n}{N}}+e^{-2\pi i \tfrac{k n}{N}})}{n} ,&p\mid N'\\
       \sum_{n=1,p|n }^\infty \frac{(p-1)(e^{2\pi i \tfrac{k n}{N}}+e^{-2\pi i \tfrac{k n}{N}})}{n}- \sum_{n=1,p\nmid n }^\infty \frac{(e^{2\pi i \tfrac{k n}{pN'}}+e^{-2\pi i \tfrac{k n}{pN'}})}{n}  ,&p\nmid N'\\
        \end{cases}\\
        =&\begin{cases}
        L^{(N')'}(0,\sigma_k),&p\mid N'\\
        L^{(N')'}(0,\sigma_k(1-\sigma_p^{-1})),&p\nmid N'.
        \end{cases} 
\end{align*}

Similarly, we have 
\begin{align*}
\sum_{\tau\in  G_{N,N'}}L^{(N)'}(-1,\tau\sigma_k)
        =&\frac{-Ni}{8\pi }\sum_{n=1 }^\infty \frac{\sum_{s\in \ker \varphi_{N,N'}}(e^{2\pi i \tfrac{sk n}{N}}-e^{-2\pi i \tfrac{sk n}{N}})}{n^2}\\
        =&\begin{cases}
       \frac{-Ni}{8\pi }\sum_{n=1,p|n }^\infty \frac{p(e^{2\pi i \tfrac{k n}{N}}-e^{-2\pi i \tfrac{k n}{N}})}{n^2} ,&p\mid N'\\
      \frac{-Ni}{8\pi } \left( \sum_{n=1,p|n }^\infty \frac{(p-1)(e^{2\pi i \tfrac{k n}{N}}-e^{-2\pi i \tfrac{k n}{N}})}{n^2}-\sum_{n=1,p\nmid n }^\infty \frac{(e^{2\pi i k \tfrac{k n}{pN'}}-e^{-2\pi i \tfrac{k n}{pN'}})}{n^2} \right)  ,&p\nmid N'\\
        \end{cases}\\
        =&\begin{cases}
        L^{(N')'}(-1,\sigma_k),&p\mid N'\\
        L^{(N')'}(-1,\sigma_k(1-p\sigma_p^{-1})),&p\nmid N'.
        \end{cases} 
\end{align*}
\end{proof}

To proceed, we observe its connection to the Mahler measure by re-writing Lemma \ref{lm} in terms of partial $L$-values.

    \begin{lem}\label{wex}For $N\geq 3$, $N\neq 4$ and $1\leq k< N/2$ with $(k,N)=1$,  $m(x_1+x_2-2\cos \tfrac{2\pi k}{N})$ lies in $\BQ$-span of \begin{equation}\label{cos}\log|2\cos\tfrac{2\pi k}{N}|=\frac{1}{2}({L^{(N/(N,2))'}}(0,\sigma_{2k/(N, 2)})-{L^{(N/(N,4))'}}(0,\sigma_{4k/(4,N)}))\end{equation} and $\sum_{\substack{1\leq k\leq N\\ (k,N)=1}} \BQ L^{(N)'}(-1,0)$ from Proposition \ref{partind}. 
    
More precisely, we have
     \[\begin{aligned}
       \delta_k m(x_1+x_2-2\cos \tfrac{2\pi k}{N})=&\frac{N-4k}{2N}({L^{(N/(N,2))'}}(0,\sigma_{2k/(2,N)})-{L^{(N/(N,4))'}}(0,\sigma_{4k/(N, 4)}))\\
        +&{\frac{4}{N/(N,2)}L^{(N/(N,2))'}(-1,\sigma_{2k/(N,2)})-\frac{2}{N/(N,4)}L^{(N/(N,4))'}(-1,\sigma_{4k/(N,4)})}.
    \end{aligned}\]
    \end{lem}
      \begin{proof}

The first part follows from Lemma \ref{Eu} and the fact that $1-p\sigma_p^{-1}$ is invertible in $\BQ[\Gal(\BQ(\mu_N)/\BQ)]$ for $(N,p)=1$.

For the second part, set $\delta_k=1$ if $1 \leq k<N/4$ and $-1$ if $N/4 \leq k <N/2$. Then Lemma Lemma \ref{lm} yields 
\[\begin{aligned}
\delta_k\cdot m(x_1+x_2-2\cos \tfrac{2\pi k}{N})=&\frac{N-4k}{N}\log|2\cos \tfrac{2\pi k}{N}|
+\frac{1}{\pi}\sum_{n=1}^\infty\frac{(-1)^{n-1}\sin \tfrac{4\pi nk}{N}}{n^2}\\
=&\frac{N-4k}{N}\sum_{n=1}^\infty\frac{(-1)^{n-1}\cos\tfrac{4\pi nk}{N}}{n}+\frac{1}{\pi}\sum_{n=1}^\infty\frac{(-1)^{n-1}\sin \tfrac{4\pi nk}{N}}{n^2},
    \end{aligned}\]where the second equality comes from the fact that  $\log|1+\zeta_N|=\frac{1}{2}\log(1+\zeta_N^2)(1+\ov{\zeta_N}^2)$ for $\zeta_N \in \mu_N$.

Note that we have
   \[\begin{aligned}
       \sum_{n=1}^\infty\frac{(-1)^{n-1}e^{i\theta n}}{n^k}
        =&\sum_{n=1}^\infty\frac{e^{i\theta n}}{n^k}-2\sum_{n=1}^\infty\frac{e^{i\theta 2n}}{2^kn^k}
    \end{aligned} ,\]
hence by definition given in Eq.\eqref{partialeq}, we conclude the claim. 
\end{proof}

We are ready to complete:
     \begin{proof}[Proof of Eq.\eqref{id:weaker}]
        Let $\CL_{-1}$, $\CL_{0}$ and $\CL$ be as defined in Section \ref{S:thmBC}.
Recall that \[\CL_{-1}^{\sharp}=\CL_{-1}=\bigoplus_{x\in S_1} \BQ L^{(N)'}(-1,x)\] where $S_1$ is a $\BQ$-basis of $\BQ[G]/<c+1>$, and $\CL_{0}^\sharp\subset \CL_{0}$ is given by 
    \[\CL_{0}^\sharp=\sum_{\zeta_N\in \mu_N^{\circ}}\BQ\log|\zeta_N+\zeta_N| \] where $\mu_N^\circ$ is the set of primitive $N$-th roots of unity. By definition, we have $\CL_{-1}^{\sharp}+ \CL_{0}^{\sharp}\subset \CL^{\sharp}$. In the following, we claim the equality via an explicit characterisation of the regulator map.

    Let $c\in H_\CM^2(Z,\partial A, \BQ(2))$ be the cohomology class constructed in \S\ref{S:randm}. It suffices to show that the image of the map
    \[H_{\text{sing},1}(Z_{\BR},\partial A_\BR,\BQ(1))\xrightarrow{\pair{\cdot,r_{\CD}(c)}}\BR\]
    is contained in $\CL_{-1}^{\sharp}+\CL_0$, where $r_{\CD}(c)\in H_{\CD}^2((Z,\partial A)_\BR,\BR(2))\simeq H_\text{sing}^1(Z,\partial A_{\BR}, \BQ(1))$ and $\pair{\ ,\ }$ denotes the natural pairing between $H_\text{sing}^1(Z,\partial A_{\BR}, \BQ(1))$ and $H_\text{sing, 1}(Z,\partial A_{\BR}, \BQ(-1))$. 
    
    Recall that we have \[Z=\bigsqcup_{\zeta_N\in \mu_N^\circ}\{z_2\in \BC\ |\ z_2\neq 0,\zeta_N\}\] and \[A=\bigsqcup_{\zeta_N\in \mu_{N}^\circ/c}[\ov{\zeta_N},\zeta_N]\] where $[\ov{\zeta_N},\zeta_N]\subset A$ is an arc $[\ov{\zeta_N},\zeta_N]=\{z_2\in \BC\ |\ |z_2|\leq 1, |\zeta_N+\ov{\zeta}_N-z_2|=1\}$. Here $\mu_N^\circ/c$ denotes the quotient of $\mu_N^\circ(=\partial A \subset \BC)$ by the complex conjugation $c$.
    We observe that $H_{\text{sing},1}(Z_{\BR},\partial A_\BR,\BQ(-1))/H_{\text{sing},0}(X_{\BR},\BQ)$ has a basis consisting of \[\frac{1}{2\pi i}[\ov{\zeta_N},\zeta_N],\quad \zeta_N\in \mu_N^\circ/c.\]

    Then by \cite{Den:97}, for any $\zeta_N=e^{2\pi i k/N}$ with $1 \leq k <N/2$ and $(k,N)=1$, we have 
\[\pair{c,\frac{1}{2\pi i}[\ov{\zeta_N},\zeta_N]}=m((x_1+x_2-2\cos \tfrac{2\pi k}{N})^*)-m(x_1+x_2-2\cos \tfrac{2\pi k}{N}).\]
Note that for $N \geq 3$, $N\neq 4$, by Jensen's formula, we have
\begin{equation}\label{onem}m((x_1+x_2-2\cos \tfrac{2\pi k}{N})^*)=\log\max\{1,|2\cos \tfrac{2\pi k}{N}|\}\end{equation}
and generically zero if $N=4$. Thus by Lemma \ref{wex}, we conclude the proof.
\end{proof}

\subsection{Explicit formulae} \label{subsec:appendix:exp}

We present several calculations that confirm our main result, in particular with an explicit description for the coefficient constant.  

Recall that we have a two-variable variant of cyclotomic polynomial defined as an irreducible monic polynomial with integral coefficients:
\[   \Psi_{N}(x)=\prod_{\substack{1\leq k<N/2\\ (k,N)=1}}(x-2\cos \tfrac{2\pi k }{N}) \]
and $f_N(x_1, x_2):=\Psi_N(x_1+x_2) \in \BZ[x_1^\pm, x_2^\pm]$.
 
    \begin{prop}\label{mainl}For $N\geq 3$, $N\neq 4$, we have $m(f_N)$ is a $\BQ$-linear combination of \[\log|2\cos \tfrac{2\pi k}{N}| \ \mbox{and } L^{(N)'}(-1,\sigma_{k}),\quad 1\leq k\leq N,\quad (k,N)=1. \] More precisely, we have 
     \[\begin{aligned}
        m(f_N)=&\sum_{\substack{1\leq k<N/4\\ (k,N)=1}}\frac{N-4k}{2N}({L^{(N/(N,2))'}}(0,\sigma_{2k/(2,N)})-{L^{(N/(N,4))'}}(0,\sigma_{4k/(N, 4)}))\\
        -&\sum_{\substack{N/4\leq k<N/2\\ (k,N)=1}}\frac{N-4k}{2N}({L^{(N/(N,2))'}}(0,\sigma_{2k/(N, 2)})-{L^{(N/(N,4))'}}(0,\sigma_{4k/(N, 4)}))\\
        +&\sum_{\substack{1\leq k<N/4\\ (k,N)=1}}(\frac{4}{N/(N,2)}L^{(N/(N,2))'}(-1,\sigma_{2k/(N,2)})-\frac{2}{N/(N,4)}L^{(N/(N,4))'}(-1,\sigma_{4k/(N,4)}))\\
        -&\sum_{\substack{N/4\leq k<N/2\\ (k,N)=1}}(\frac{4}{N/(N,2)}L^{(N/(N,2))'}(-1,\sigma_{2k/(N,2)})-\frac{2}{N/(N,4)}L^{(N/(N,4))'}(-1,\sigma_{4k/(N,4)}))
    \end{aligned}\]
    \end{prop}
    \begin{proof}This is a direct consequence of Lemma \ref{wex}. We remark that when $N=4$, we have $m(f_N)=0$. \end{proof}

This now enables us to express $m(f_N)$ as a $\BQ$-linear combination of $L'(0,\chi)$ for an even $\chi$ modulo $N$ and $L'(-1,\chi)$ for an odd $\chi$ modulo $N$.
  
      \begin{cor}\label{maincal}
    We have  \[m(f_N)=\sum_{\chi\in \wh{(\BZ/N\BZ)^\times}}\mu_{N,\chi}\delta_{N,\chi} L^{(N_\chi)'} \left(\frac{\chi(-1)-1}{2},\chi \right),\]where $N_\chi=\begin{cases}N/4,&\quad 4\parallel N, \text{$\chi$ even, and $\cond(\chi)\mid N/4$}\\ 
    N,&\quad \text{otherwise,}\end{cases}$ 
    \[
    \mu_{N,\chi}
    =
    \begin{cases}
        \sum_{\substack{1\leq k<N/4\\ (k,N)=1}} (N-4k) {\chi^{-1}(k )} -\sum_{\substack{N/4\leq k<N/2\\ (k,N)=1}}(N-4k) {\chi^{-1}(k )},\quad& \text{$\chi$ is even}\\
        \sum_{\substack{1\leq k<N/4\\ (k,N)=1}}\chi^{-1}(k)-\sum_{\substack{N/4\leq k<N/2\\ (k,N)=1}}\chi^{-1}(k),\quad&\text{$\chi$ is odd}\\
    \end{cases}\]
and 
\[
    \delta_{N,\chi}
    =
    \begin{cases}
       \frac{\chi^{-1}(2)}{2N\varphi(N)}(1-\chi^{-1}(2)) ,&\quad \text{$2\nmid N$ and $\chi$ is even}\\
\frac{4\chi^{-1}(2)}{N\varphi(N)}(1-\chi^{-1}(2)/{2}),&\quad \text{$2\nmid N$ and $\chi$ is odd}\\
\frac{-\chi^{-1}(2)}{2N\varphi(N)},&\quad \text{$2\parallel N$ and $\chi$ is even and $\cond(\chi)\mid N/2$}\\
\frac{-4\chi^{-1}(2)}{N\varphi(N)},&\quad \text{$2\parallel N$ and $\chi$ is odd and $\cond(\chi)\mid N/2$}\\
\frac{-\chi(2)}{2\varphi(N/4)N},&\quad \text{$4\parallel N$, $\chi$ is even and $\cond(\chi)\mid N/4$}\\
\frac{-4\chi(2)}{\varphi(N/4)N/4}\frac{1}{1-2\chi(2)},&\quad \text{$4\parallel N$, $\chi$ is odd and $\cond(\chi)\mid N/4$}\\
\frac{1}{4N\varphi(N/4)},&\quad \text{$8\mid N$, $\chi$ is even, $N/4\nmid \cond(\chi)\mid N/2$}\\
\frac{-1}{4N\varphi(N/4)},&\quad \text{$8\mid N$, $\chi$ is even, $\cond(\chi)\mid N/4$}\\
\frac{1}{\varphi(N/4)N/4},&\quad \text{$8\mid N$, $\chi$ is odd, $N/4\nmid \cond(\chi)\mid N/2$}\\
\frac{-1}{\varphi(N/4)N/4},&\quad \text{$8\mid N$, $\chi$ is odd, $\cond(\chi)\mid N/4$}\\
0,&\quad \text{otherwise}.
    \end{cases}\]
\end{cor}

 \begin{proof} 
 Here we briefly show the case $(N,4)=2$, the other cases proceed similarly.  
Write $N=2M$ with an odd $M$. Let $\chi\in \wh{(\BZ/ {M} \BZ)^\times}$. Then by Lemma \ref{Eu} and Proposition \ref{mainl}, 
     \[\begin{aligned}
        &m(f_N)\\
       = &\sum_{\substack{ \chi: \text{even}}}(\sum_{\substack{1\leq k<N/4\\ (k,N)=1}} (N-4k) {\chi^{-1}(k )} -\sum_{\substack{N/4\leq k<N/2\\ (k,N)=1}}(N-4k) {\chi^{-1}(k )})\frac{1-\chi^{-1}(2)}{2N\varphi(N)}L^{(M)'}(0,\chi)\\
        +&\sum_{\substack{\chi:\text{odd}}}(\sum_{\substack{1\leq k<N/4\\ (k,N)=1}}\chi^{-1}(k)-\sum_{\substack{N/4\leq k<N/2\\ (k,N)=1}}\chi^{-1}(k))\frac{8-4\chi^{-1}(2)}{N\varphi(N)} L^{(M)'}(-1,\chi)\\
           = &\sum_{\substack{ \chi: \text{even}}}(\sum_{\substack{1\leq k<N/4\\ (k,N)=1}} (N-4k) {\chi^{-1}(k )} -\sum_{\substack{N/4\leq k<N/2\\ (k,N)=1}}(N-4k) {\chi^{-1}(k )})\frac{-\chi^{-1}(2)}{2N\varphi(N)}L^{(N)'}(0,\chi)\\
       + &\sum_{\substack{\chi: \text{odd}}}(\sum_{\substack{1\leq k<N/4\\ (k,N)=1}}\chi^{-1}(k)-\sum_{\substack{N/4\leq k<N/2\\ (k,N)=1}}\chi^{-1}(k))\frac{-4\chi^{-1}(2)}{N\varphi(N)}L^{(N)'}(-1,\chi).
    \end{aligned} \]
        \end{proof}

Combining the results on $m(f_N^*)$ (cf.~Equation~\eqref{onem}) and Corollary \ref{maincal}, we have:

\begin{cor}\label{exp}
 We have  \[m(f_N^*)-m(f_N)=\sum_{\chi\in \wh{(\BZ/N\BZ)^\times}}r_{N,\chi}L^{(N_\chi)'}\left(\frac{\chi(-1)-1}{2},\chi \right)\]where $N_\chi=\begin{cases}N/4,&\quad 4\parallel N, \text{$\chi$ even and $\cond(\chi)\mid N/4$}\\ 
    N,&\quad \text{otherwise}\end{cases}$, $r_{N,\chi}=\epsilon_{N,\chi}\delta_{N,\chi}$ with $\delta_{N,\chi}$ is as in Corollary \ref{maincal} and 
    \[\displaystyle
    \epsilon_{N,\chi}
    =
    \begin{cases}
        \substack{\sum_{\substack{1\leq k<N/2,(k,N)=1\\ |4k-N|>N/3 }}N{\chi^{-1}(k )} -\sum_{\substack{1\leq k<N/4\\ (k,N)=1}} (N-4k) {\chi^{-1}(k )} +\sum_{\substack{N/4\leq k<N/2\\ (k,N)=1}}(N-4k) {\chi^{-1}(k )}}, \quad& \text{$\chi$ is even}\\
        -\sum_{\substack{1\leq k<N/4\\ (k,N)=1}}\chi^{-1}(k)+\sum_{\substack{N/4\leq k<N/2\\ (k,N)=1}}\chi^{-1}(k), \quad&\text{$\chi$ is odd.}
    \end{cases}\]
\end{cor}

\medskip 

Finally, we conclude the non-zeroness statement related to $r_{N,\chi}$. Let $\chi$ be an odd Dirichlet character of the conductor $c\geq 3$. 
Note that it is enough to assume that $c$ is odd or $4 \mid c$.

\begin{lem}\label{nzo}Let $N=\begin{cases}
c,\quad &\text{if $c$ is odd},\\ 
2c,\quad &\text{if $4 \mid c$}.
\end{cases}$.
We have 
\[\mu_{N,\chi}=\begin{cases}
2\chi(2)(1-\frac{\chi(2)}{2})L(0,\chi^{-1}),\quad &\text{if $c$ is odd}\\
4L(0,\chi^{-1}),\quad &\text{if $4 \mid c$},
\end{cases}\]which is in particular non-zero.
\end{lem}
\begin{proof}Let $S_1=\sum_{\substack{1\leq k<N/4\\ (k,N)=1}}\chi^{-1}(k)$ and $S_2=-\sum_{\substack{N/4\leq k<N/2\\ (k,N)=1}}\chi^{-1}(k)$.
First consider $c=N$ is odd.
Then by \cite[Lem.~4]{Xu:Zhang}, one has
\[S_1=\frac{2+\chi(2)-\chi(4)}{2}L(0,\chi^{-1}),\]where $L(0,\chi^{-1})=-\frac{1}{c}\sum_{k=1}^c k \chi^{-1}(k)$.
Then by \cite[Lem.~3]{Xu:Zhang}, one has
 \[S_1-S_2=(2-\chi(2))L(0,\chi^{-1}).\]

Hence \[
\begin{aligned}
\mu_{N,\chi}=&2\chi(2)(1-\frac{\chi(2)}{2})L(0,\chi^{-1}).
\end{aligned}\] 
If $c$ is even, we have $N=2c$ and 
$\mu_{N,\chi}=2S_1$.
The function $f=1_{1\leq k\leq  c/2-1}$ on $\BZ/c\BZ$ has a Fourier expansion 
\[\begin{aligned}
f(x)=&\sum_{k=1}^c \left(\frac{1}{c}\sum_{a=1}^{c/2-1}e^{\frac{-2\pi i a k}{c}}\right)e^{\frac{2\pi i k x}{c}} \\
\end{aligned}\]
and thus \[\begin{aligned}
 \frac{1}{c}\sum_{i=1}^{c/2}\chi^{-1}(c)=&{ \sum_{k=1}^c \left(\frac{1}{c}\sum_{a=1}^{c/2-1}e^{\frac{-2\pi i a k}{c}}\right)(\frac{1}{c}\sum_{x=1}^c e^{\frac{2\pi i k x}{c}}\chi^{-1}(x))}\\
 =& \frac{\tau(\chi^{-1})}{c}\frac{2}{c}\sum_{k=1, (k,c)=1}^c  \left(\frac{1}{1-e^{\frac{-2\pi i k }{c}}}\right)\chi(k) .\\
\end{aligned}\]

Note that for $1\leq a< c$,
\[\frac{a}{c}=\frac{c-1}{2c}-\frac{1}{c}\sum_{k=1}^{c-1}\frac{e^{\frac{2\pi i ka}{c}}}{1-e^{\frac{-2\pi i k}{c}}} \]
in turn 
\[\begin{aligned}
L(0,\chi^{-1})=&-\frac{1}{c}\sum_{a=1}^c a\chi^{-1}(a)\\
=&\frac{\tau(\chi^{-1})}{c} \sum_{k=1, (k,c)=1}^{c}\frac{1}{1-e^{\frac{-2\pi i k}{c}}}\chi(k).
\end{aligned}\]
Thus we have
$\sum_{i=1}^{c/2}\chi^{-1}(c)=2L(0,\chi^{-1})$.
\end{proof}

\bibliographystyle{alpha}
\bibliography{Dirichlet}

@article {guil:mar,
    AUTHOR = {Guilloux, Antonin and March\'e, Julien},
     TITLE = {Volume function and {M}ahler measure of exact polynomials},
   JOURNAL = {Compos. Math.},
  FJOURNAL = {Compositio Mathematica},
    VOLUME = {157},
      YEAR = {2021},
    NUMBER = {4},
     PAGES = {809--834},
      ISSN = {0010-437X,1570-5846},
   MRCLASS = {11R06 (14T20 19C99 57K14)},
  MRNUMBER = {4247573},
MRREVIEWER = {Detchat\ Samart},
       DOI = {10.1112/s0010437x21007016},
       URL = {https://doi.org/10.1112/s0010437x21007016},
}

@misc{calegari2024,
      title={The linear independence of $1$, $\zeta(2)$, and $L(2,\chi_{-3})$}, 
      author={Frank Calegari and Vesselin Dimitrov and Yunqing Tang},
      year={2024},
      eprint={2408.15403},
      archivePrefix={arXiv},
      primaryClass={math.NT},
      url={https://arxiv.org/abs/2408.15403}, 
}

@article {coleman1988p,
    AUTHOR = {Coleman, Robert and de Shalit, Ehud},
     TITLE = {{$p$}-adic regulators on curves and special values of
              {$p$}-adic {$L$}-functions},
   JOURNAL = {Invent. Math.},
  FJOURNAL = {Inventiones Mathematicae},
    VOLUME = {93},
      YEAR = {1988},
    NUMBER = {2},
     PAGES = {239--266},
      ISSN = {0020-9910},
   MRCLASS = {11G05 (11G15 11Q25 11S40 19F27)},
  MRNUMBER = {948100},
MRREVIEWER = {V. Kumar Murty},
       DOI = {10.1007/BF01394332},
       URL = {https://doi.org/10.1007/BF01394332},
}

@incollection {beilinson1984higher,
    AUTHOR = {Beilinson, A. A.},
     TITLE = {Higher regulators and values of {$L$}-functions},
 BOOKTITLE = {Current problems in mathematics, {V}ol. 24},
    SERIES = {Itogi Nauki i Tekhniki},
     PAGES = {181--238},
 PUBLISHER = {Akad. Nauk SSSR, Vsesoyuz. Inst. Nauchn. i Tekhn. Inform.,
              Moscow},
      YEAR = {1984},
   MRCLASS = {11R42 (11G40 11G45 11R70 14C35 18F25 19F27)},
  MRNUMBER = {760999},
MRREVIEWER = {Daniel\ R.\ Grayson},
}

@article{nekovar1994beilinson,
  title={Beilinson’s conjectures},
  author={Nekov{\'a}r, Jan},
  journal={Motives (Seattle, WA, 1991)},
  volume={55},
  number={Part 1},
  pages={537--570},
  year={1994}
}

@book{friedlander2005handbook,
  title={Handbook of K-theory},
  author={Friedlander, Eric and Grayson, Daniel R},
  year={2005},
  publisher={Springer Science \& Business Media}
}

@article {Boyd:Rod,
    AUTHOR = {Boyd, David W. and Rodriguez-Villegas, Fernando},
     TITLE = {Mahler's measure and the dilogarithm. {I}},
   JOURNAL = {Canad. J. Math.},
  FJOURNAL = {Canadian Journal of Mathematics. Journal Canadien de
              Math\'ematiques},
    VOLUME = {54},
      YEAR = {2002},
    NUMBER = {3},
     PAGES = {468--492},
      ISSN = {0008-414X,1496-4279},
   MRCLASS = {11G55 (11R09 11R42 19F27)},
  MRNUMBER = {1900760},
MRREVIEWER = {Jan\ Nekov\'a\v r},
       DOI = {10.4153/CJM-2002-016-9},
       URL = {https://doi.org/10.4153/CJM-2002-016-9},
}

@misc{Bertin,
      title={An exact family of bivariate polynomials and Variants of Chinburg's Conjectures}, 
      author={Marie-José Bertin and Mahya Mehrabdollahei},
      year={2025},
      eprint={2407.20634},
      archivePrefix={arXiv},
      primaryClass={math.NT},
      url={https://arxiv.org/abs/2407.20634}, 
}

@incollection {Chinburg,
    AUTHOR = {Chinburg, Ted},
     TITLE = {Mahler measures and derivatives of $L$-functions at non-positive integers},
 BOOKTITLE = {Unpublished work},
      YEAR = {1984},
}

@incollection {Ville,
    AUTHOR = {Villegas, F. Rodriguez},
     TITLE = {Modular {M}ahler measures. {I}},
 BOOKTITLE = {Topics in number theory ({U}niversity {P}ark, {PA}, 1997)},
    SERIES = {Math. Appl.},
    VOLUME = {467},
     PAGES = {17--48},
 PUBLISHER = {Kluwer Acad. Publ., Dordrecht},
      YEAR = {1999},
      ISBN = {0-7923-5583-0},
   MRCLASS = {11G40 (11R06 19F27)},
  MRNUMBER = {1691309},
MRREVIEWER = {Jan\ Nekov\'a\v r},
}

@article {Boyd,
    AUTHOR = {Boyd, David W.},
     TITLE = {Mahler's measure and special values of {$L$}-functions},
   JOURNAL = {Experiment. Math.},
  FJOURNAL = {Experimental Mathematics},
    VOLUME = {7},
      YEAR = {1998},
    NUMBER = {1},
     PAGES = {37--82},
      ISSN = {1058-6458,1944-950X},
   MRCLASS = {11G40 (11R06 11Y35)},
  MRNUMBER = {1618282},
MRREVIEWER = {Jan\ Nekov\'a\v r},
       URL = {http://projecteuclid.org/euclid.em/1047674271},
}

@article {LSW:90,
    AUTHOR = {Lind, Douglas and Schmidt, Klaus and Ward, Tom},
     TITLE = {Mahler measure and entropy for commuting automorphisms of
              compact groups},
   JOURNAL = {Invent. Math.},
  FJOURNAL = {Inventiones Mathematicae},
    VOLUME = {101},
      YEAR = {1990},
    NUMBER = {3},
     PAGES = {593--629},
      ISSN = {0020-9910},
   MRCLASS = {22D40 (28D20)},
  MRNUMBER = {1062797},
MRREVIEWER = {Meir Smorodinsky},
       DOI = {10.1007/BF01231517},
       URL = {https://0-doi-org.pugwash.lib.warwick.ac.uk/10.1007/BF01231517},
}

@article {Smyth:81,
    AUTHOR = {Smyth, C. J.},
     TITLE = {On measures of polynomials in several variables},
   JOURNAL = {Bull. Austral. Math. Soc.},
  FJOURNAL = {Bulletin of the Australian Mathematical Society},
    VOLUME = {23},
      YEAR = {1981},
    NUMBER = {1},
     PAGES = {49--63},
      ISSN = {0004-9727},
   MRCLASS = {10K50 (12A15)},
  MRNUMBER = {615132},
MRREVIEWER = {G\'{e}rard Rauzy},
       DOI = {10.1017/S0004972700006894},
       URL = {https://0-doi-org.pugwash.lib.warwick.ac.uk/10.1017/S0004972700006894},
}

@article {Bes:Den,
    AUTHOR = {Besser, Amnon and Deninger, Christopher},
     TITLE = {{$p$}-adic {M}ahler measures},
   JOURNAL = {J. Reine Angew. Math.},
  FJOURNAL = {Journal f\"{u}r die Reine und Angewandte Mathematik. [Crelle's
              Journal]},
    VOLUME = {517},
      YEAR = {1999},
     PAGES = {19--50},
      ISSN = {0075-4102},
   MRCLASS = {11G40 (11S80)},
  MRNUMBER = {1728549},
MRREVIEWER = {Bruno Chiarellotto},
       DOI = {10.1515/crll.1999.093},
       URL = {https://0-doi-org.pugwash.lib.warwick.ac.uk/10.1515/crll.1999.093},
}

@article {Den:97,
    AUTHOR = {Deninger, Christopher},
     TITLE = {Deligne periods of mixed motives, {$K$}-theory and the entropy
              of certain {${\bf Z}^n$}-actions},
   JOURNAL = {J. Amer. Math. Soc.},
  FJOURNAL = {Journal of the American Mathematical Society},
    VOLUME = {10},
      YEAR = {1997},
    NUMBER = {2},
     PAGES = {259--281},
      ISSN = {0894-0347},
   MRCLASS = {11G40 (19E08 19F27 28D20)},
  MRNUMBER = {1415320},
MRREVIEWER = {Jan Nekov\'{a}\v{r}},
       DOI = {10.1090/S0894-0347-97-00228-2},
       URL = {https://0-doi-org.pugwash.lib.warwick.ac.uk/10.1090/S0894-0347-97-00228-2},
}

@book{CT,
  title={L-functions and Arithmetic},
  author={Coates, John and Taylor, Martin J},
  volume={153},
  year={1991},
  publisher={Cambridge University Press}
}

@article {Xu:Zhang,
    AUTHOR = {Xu, Zhefeng and Zhang, Wenpeng},
     TITLE = {Some identities involving the Dirichlet {$L$}-function},
   JOURNAL = {Acta Arith.},
  FJOURNAL = {},
    VOLUME = {},
      YEAR = {2007},
     PAGES = {157--166},
      ISSN = { },
   MRCLASS = { },
  MRNUMBER = { },
MRREVIEWER = { },
       DOI = { },
       URL = { },
}

\end{document}